\documentclass[11pt,onecolumn]{article}
\usepackage{amsmath,amsthm,amssymb,amsfonts,epsfig,graphicx}
\usepackage{subcaption}
\usepackage{color}  
\usepackage{url}
\usepackage{array,multirow}
\usepackage{bm}
\usepackage{xcolor}
\usepackage{pifont}
\usepackage{pdflscape}
\usepackage{rotating}
\usepackage{comment}
\usepackage{hyperref}

\usepackage{stmaryrd}
\graphicspath{{./}, {figures/}}

\usepackage{enumitem}
\usepackage[vlined,linesnumbered,ruled,resetcount]{algorithm2e}

\usepackage[left=1.0in,top=0.5in,right=1.05in,bottom=1.0in,nohead,textheight=10in,footskip=0.3in]{geometry}

\newcommand{\ALcode}{\texttt{AL}}
\newcommand{\ALcodeL}{\texttt{AL-L}}
\newcommand{\ALcodeC}{\texttt{AL-C}}
\newcommand{\CLASSICcode}{\texttt{CLASSIC}}
\newcommand{\CLASSICcodeL}{\texttt{CLASSIC-L}}
\newcommand{\CLASSICcodeC}{\texttt{CLASSIC-C}}
\newcommand{\cRATIOcode}{\texttt{cRATIO}}
\newcommand{\cLOGcode}{\texttt{cLOG}}
\newcommand{\cSORTcode}{\texttt{cSORT}}

\usepackage{glossaries}
\newglossaryentry{AL}{name=AL, description={ArcLength}}
\newglossaryentry{LAL}{name=LAL, description={Lifted ArcLength}}
\newglossaryentry{SVD}{name=SVD, description={Singular Value Decomposition}}
\newglossaryentry{NAG}{name=NAG, description={Numerical Algebraic Geometry}}

\newtheorem{theorem}{Theorem}
\newtheorem{corollary}[theorem]{Corollary}

\newtheorem{assumption}{Assumption}
\newtheorem{lemma}[theorem]{Lemma}

\newtheorem{proposition}[theorem]{Proposition}
\newtheorem{remark}{Remark}

\DeclareMathOperator{\inter}{int}

\SetKwInput{kwMyRepeat}{Repeat}

\begin{document}

\def\myparagraph#1{\vspace{2pt}\noindent{\bf #1~~}}



\def\DeltaCeil{{\lceil\Delta\rceil}}
\def\TwoDeltaCeil{{\lceil 2\Delta\rceil}}
\def\OnePointFiveDeltaCeil{{\lceil 3\Delta/2\rceil}}

\long\def\ignore#1{}
\def\myps[#1]#2{\includegraphics[#1]{#2}}
\def\etal{{\em et al.}}
\def\Bar#1{{\bar #1}}
\def\br(#1,#2){{\langle #1,#2 \rangle}}
\def\setZ[#1,#2]{{[ #1 .. #2 ]}}
\def\Pr#1{\mbox{\tt Pr}\left[{#1}\right]}
\def\REACHED{\mbox{\tt REACHED}}
\def\AdjustFlow{\mbox{\tt AdjustFlow}}
\def\GetNeighbors{\mbox{\tt GetNeighbors}}
\def\true{\mbox{\tt true}}
\def\false{\mbox{\tt false}}
\def\Process{\mbox{\tt Process}}
\def\ProcessLeft{\mbox{\tt ProcessLeft}}
\def\ProcessRight{\mbox{\tt ProcessRight}}
\def\Add{\mbox{\tt Add}}

\newcommand{\eqdef}{{\stackrel{\mbox{\tiny \tt ~def~}}{=}}}

\def\setof#1{{\left\{#1\right\}}}
\def\suchthat#1#2{\setof{\,#1\mid#2\,}} 
\def\event#1{\setof{#1}}
\def\q={\quad=\quad}
\def\qq={\qquad=\qquad}
\def\calA{{\cal A}}
\def\calB{{\cal B}}
\def\calC{{\cal C}}
\def\calD{{\cal D}}
\def\calE{{\cal E}}
\def\calK{{\cal K}}
\def\calG{{\cal G}}
\def\calI{{\cal I}}
\def\calF{{\cal F}}
\def\calH{{\cal H}}
\def\calL{{\cal L}}
\def\calN{{\cal N}}
\def\calM{{\cal M}}
\def\calP{{\cal P}}
\def\calR{{\cal R}}
\def\calS{{\cal S}}
\def\calT{{\cal T}}
\def\calU{{\cal U}}
\def\calV{{\cal V}}
\def\calO{{\cal O}}
\def\calX{{\cal X}}

\def\U{{\mathbb U}}
\def\E{{\calE}}
\def\T{{\calF}}
\def\XX(#1){{{#1}^\downarrow}}

\def\calY{{\cal Y}}
\def\calZ{{\cal Z}}
\def\s{\footnotesize}
\def\calNG{{\cal N_G}}
\def\psfile[#1]#2{}
\def\psfilehere[#1]#2{}
\def\epsfw#1#2{\includegraphics[width=#1\hsize]{#2}}
\def\assign(#1,#2){\langle#1,#2\rangle}
\def\edge(#1,#2){(#1,#2)}
\def\VS{\calV^s}
\def\VT{\calV^t}
\def\slack(#1){\texttt{slack}({#1})}
\def\barslack(#1){\overline{\texttt{slack}}({#1})}
\def\NULL{\texttt{NULL}}
\def\PARENT{\texttt{PARENT}}
\def\GRANDPARENT{\texttt{GRANDPARENT}}
\def\TAIL{\texttt{TAIL}}
\def\HEADORIG{\texttt{HEAD$\_\:$ORIG}}
\def\TAILORIG{\texttt{TAIL$\_\:$ORIG}}
\def\HEAD{\texttt{HEAD}}
\def\CURRENTEDGE{\texttt{CURRENT$\!\_\:$EDGE}}

\def\unitvec(#1){{{\bf u}_{#1}}}
\def\uvec{{\bf u}}
\def\vvec{{\bf v}}
\def\Nvec{{\bf N}}
\def\r{{\bf r}}

\newcommand{\bg}{\mbox{$\bf g$}}
\newcommand{\bh}{\mbox{$\bf h$}}

\newcommand{\bx}{\mbox{$x$}}
\newcommand{\by}{\mbox{\boldmath $y$}}
\newcommand{\bz}{\mbox{\boldmath $z$}}
\newcommand{\bu}{\mbox{\boldmath $u$}}
\newcommand{\bv}{\mbox{\boldmath $v$}}
\newcommand{\bw}{\mbox{\boldmath $w$}}
\newcommand{\bvarphi}{\mbox{\boldmath $\varphi$}}
\newcommand{\balpha}{\mbox{\boldmath $\alpha$}}

\newcommand\myqed{{}}

\newcommand{\IBFSFS}{{IBFS$^{\mbox{~\!\tiny FS}}$}}
\newcommand{\IBFSAL}{{IBFS$^{\mbox{~\!\tiny AL}}$}}
\newcommand{\BKFS}{{BK$^{\mbox{~\!\tiny FS}}$}}
\newcommand{\BKAL}{{BK$^{\mbox{~\!\tiny AL}}$}}

\newcommand{\XL}{{X_{\le L}}}

\def\br#1{{\llbracket #1 \rrbracket}}

\newcommand{\mik}[1]{{\textcolor{blue}{#1}}}
\newcommand{\vnk}[1]{{\textcolor{olive}{#1}}}

\newcommand{\jef}[1]{{\textcolor{red}{#1}}}

\newcommand{\newjef}[1]{
    \noindent\textcolor{olive}{\ding{43} \textbf{New:}} \textcolor{olive}{#1}
}

\title{\Large\bf  }
\author{}
\title{
\Large\bf  \vspace{-20pt} 
A superlinear arclength endgame \\ for corank-one singular polynomial systems
}
\author{
  Mikhail Karapetyants \qquad Vladimir Kolmogorov \qquad Jeferson Zapata \\
  \normalsize Institute of Science and Technology Austria (ISTA) \\
  \texttt{\{mikhail.karapetyants,vnk,jeferson.zapata\}@ist.ac.at}
}


\date{}
\maketitle
\begin{abstract}
In Numerical Algebraic Geometry (NAG) isolated solutions of polynomial systems are usually computed by tracking a solution curve defined by a homotopy equation. 
The tracking problem becomes especially challenging close to a singular root (the ``endgame'' regime). 
Existing approaches include power series endgames, Cauchy endgames, and various methods that regularize the system via dual-space-based {\em deflation}.

We introduce a new ``Arclength Endgame'' for corank-1 systems  which combines the idea of the classical {\em pseudo-arclength continuation method} 
with the estimation of the Puiseux series of the curve. We formally prove that it has an R-superlinear rate of convergence in some neighborhood of the root. 
The method uses only evaluations of the system and its Jacobian. In contrast, previous homotopy techniques with proven superlinear convergence 
(such as deflation) require computing additional derivatives of the system, or need additional nondegeneracy assumptions involving higher-order derivatives of the system at the singular root.

A key step in our approach (as well as in the standard power series endgame) is estimating the ratio $k_1/c$ of the Puiseux series of the curve. We present a new method for that, and discuss its advantages over previous methods.
\end{abstract}


\section{Introduction}\label{section:Introduction}
We consider the problem of numerically solving a system of equations $f(z)=0$
with a zero-dimensional set of solutions.
Here $f$ is an analytic mapping $\mathbb C^n\rightarrow\mathbb C^n$.
A standard approach in Numerical Algebraic Geometry (NAG) for tackling this problem is as follows.
First, one constructs a homotopy function $h(z,t)=(1-t)f(z)+t g(z)$
where $g$ is a polynomial system with easily computable roots.
Consider one such root $z_{\texttt{root}}$.
If $g$ is chosen generically then there exists a unique smooth function $z:(0,1]\rightarrow \mathbb C^n$
with $z(1)=z_{\texttt{root}}$ and $h(z(t),t)=0$ for all $t\in(0,1]$. Furthermore, if $z((0,1])$ is bounded then the limit $z^\ast=z(0)=\lim_{t\rightarrow 0}z(t)$ exists
and is a root of~$f$. 
If $g$ is chosen to have sufficiently many roots then every root of $f$ will be covered with probability~1 \cite[Theorem 8.4.1]{SommeseWampler}.

By differentiating equation $h(z(t),t)=0$ with respect to $t$ one obtains Davidenko ODE:
\begin{equation}
\dot z(t) = - h_z (z(t),t)^{-1} h_t(z(t),t)
\end{equation}
We now need to track curve $z(t)$ by numerically solving this ODE.
 
This problem becomes especially challenging when $t$ approaches zero and $z^\ast=z(0)$ is a singular solution $z^\ast$,
i.e.\ the Jacobian $f_z(z^\ast)$ is singular.
This regime, referred to as the ``endgame'', constitutes the main focus of the present work. The following assumptions are maintained throughout the manuscript:

\begin{assumption}\label{assumption:main}
\textup{(a)} $h:\mathbb C^{n+1}\rightarrow\mathbb C^n$ is a mapping which is analytic at $(z^\ast,0)$, and $f(z)=h(z,0)$. \\
\textup{(b)} Point $z^\ast\in\mathbb C^n$ is an isolated solution of $f(z)$. We denote $x^\ast=(z^\ast,0)$, $J^\ast= h_z(x^\ast)=f_z(z^\ast)$ and $h^\ast_t=h_t(x^\ast)$. 
We also let $\kappa=n-{\tt rank}(J^\ast)$ be the corank of $J^\ast$. \\
\textup{(c)} ${\tt rank}(J^\ast)<n$, i.e.\ $z^\ast$ is a singular solution of $f$. \\
\textup{(d)} $h_t^\ast$ is linearly independent of columns in $J^\ast$, i.e.\ $${\tt rank}([J^\ast|h^\ast_t])={\tt rank}(J^\ast)+1.$$ \\
\textup{(e)} There exists a finite set $\Pi$ of formal Puiseux series of the form
\begin{equation}\label{eq:Puiseux}
z(t)=z^\ast + \sum_{j=1}^\infty a_j t^{j/c} = z^\ast + a_{k_1} t^{k_1/c} + a_{k_2} t^{k_2/c} + \ldots
\end{equation}
where $c\in\mathbb N$ is minimal and $1\le k_1 < k_2 < k_3 \ldots$ is an increasing sequence of integers corresponding to non-zero coefficients $a_j$
such that:  \\
\textup{(i)} each series\footnote{
When writing $z(\cdot)\in \Pi$, we will assume with some abuse of notation
that we have chosen not only the formal series but also the specific branch $t\mapsto t^{1/c}$ used in~\eqref{eq:Puiseux}
(unless noted otherwise).
}
 $z(\cdot)\in \Pi$ is convergent in some neighborhood $\Omega_t\subseteq\mathbb C$ of~$0$ and 
satisfies  $h(z(t),t)=0$ for $t\in \Omega_t$; \\
\textup{(ii)} for every open set $U\subseteq \Omega_t$ and every $z(\cdot)\in\Pi$, function $\phi(t)=\det [h_z(z(t),t)]$ is not identically zero on $U$; \\
\textup{(iii)} there exists neighborhood $\Omega_z$ of $z^\ast$ such that for every $(z,t)\in\Omega_z\times\Omega_t$ with $h(z,t)=0$
there exists $z(\cdot)\in\Pi$ with $z(t)=z$.
\end{assumption}
\noindent 
The Puiseux parametrization in \textup{(e)} is a standard consequence of the local
parametrization theorem for complex analytic curves
(and of the assumption that $z^\ast$ is isolated root of $f$).
The additional
transversality and punctured-regularity requirements in \textup{(d)} and \textup{(e-ii)} are
generic for the usual coefficient homotopies used in numerical algebraic
geometry; in practice they are enforced by choosing a generic start system
and applying the standard gamma trick.

\subsection{Corank-1 problems}
Our first contribution is the following result.

\begin{theorem}\label{th:gap1:main}
Suppose that $\kappa=1$. There exist a neighborhood $\Omega$ of $x^\ast$, constants $\alpha>1,\beta_{\min}>0$ 
and an algorithm that takes point $x_\circ=(z_\circ,t_\circ)\in\mathbb C^{n+1}$ and parameter $\beta$, and
does the following:
if $\beta>\beta_{\min}$, $x_\circ\in \Omega$ and $\|h(x_\circ)\|\le |t_\circ|^\beta$ 
then it produces a (possibly infinite) sequence of points $x_1,\ldots,x_K$
such that \\
(i) $\|x_k-x^\ast\|\le \|x_\circ-x^\ast\|^{\alpha^k}$ for each $k\in[K]$,
and (ii) if $K<\infty$, then the last point $x=x_K=(z,t)$ satisfies $\|h(x)\|\le |t|^\beta$.
It uses 
$O(\log (1+\beta)+K)$ 
evaluations of  function $h(\cdot)$ and its Jacobian.
\end{theorem}

A recursive application of the algorithm in Theorem~\ref{th:gap1:main} immediately gives an algorithm with an R-superlinear convergence rate.
\begin{corollary}
Suppose that $\kappa=1$. There exist a neighborhood $\Omega$ of $x^\ast$, constants $\alpha>1,\beta_{\min}>0$ 
and an algorithm that takes point $x_\circ=(z_\circ,t_\circ)\in\mathbb C^{n+1}$ and parameter $\beta$, and
does the following:
if $\beta>\beta_{\min}$, $x_\circ\in \Omega$ and $\|h(x_\circ)\|\le |t_\circ|^\beta$ 
then it produces a sequence of points $x_1,x_2,\ldots$
such that $\|x_k-x^\ast\|\le \|x_\circ-x^\ast\|^{\alpha^k}$ for each $k\ge 1$.
Computing each subsequent point uses 
$O(\log (1+\beta))$ 
evaluations of  function $h(\cdot)$ and its Jacobian.
\end{corollary}

The algorithm in Theorem~\ref{th:gap1:main} is achieved by combining two steps.
\begin{itemize}
\item Predictor phase: estimate coefficients $z^\ast,k_1/c,a_{k_1}$ of a Puiseux series $z(\cdot)\in\Pi$, 
then evaluate the obtained approximation of $z(\cdot)$ at $t=0$ obtaining predictor point $\hat x=(\hat z,0)$. 
We use a new rule for estimating $k_1/c$, and formally analyze the accuracy of the resulting approximation.
\item Corrector phase: augment the system $h(x)=0$ with a new linear equation given by a hyperplane passing through $\hat x$,
whose normal is tangent to
the solution curve.
Solve the augmented system by applying several steps of Newton's method.
\end{itemize}
The idea of adding an extra hyperplane is common in the {\em (pseudo)-arclength} continuation methods \cite{wempner1971discrete,riks1972application,keller1977numerical}.
Our scheme differs in the choice of the step size.
To our knowledge, existing arclength methods control the step size via parameter $\Delta s$ 
representing the Euclidean length along the solution curve.
We are not aware of methods that explicitly combine an arclength method together with a Puiseux series-aware predictor.

Due to this connection, we call our method the {\em arclength endgame}, even though it does not use the Euclidean length of the curve in any way.

\myparagraph{Related work}
Below we discuss papers that give an explicit superlinear convergence rate when $t\rightarrow 0$.
One approach is the {\em deflation} technique
\cite{ojika1983deflation,LeykinVerscheldeZhao2006,
LeykinVerscheldeZhao2008,HauensteinMourrainSzanto2017},
which augments or modifies the original system so that, after one or more
deflation steps, the original singular solution is represented by a {\bf nonsingular} root
of the deflated system.
Different deflation constructions may introduce additional equations,
auxiliary variables, or both.
Once the resulting root is nonsingular, the classical Newton method has a
local quadratic convergence rate.
The deflated system can be significantly larger than the original system,
and its size may depend on the multiplicity of $z^\ast$.
For corank-one systems, specialized techniques with quadratic convergence
were proposed in~\cite{LiZhi2012,LZ12SINUM,LiZhi2022}.


Note that these techniques use
derivatives of the input system beyond the first order.
For example, the methods in~\cite{LiZhi2012,LZ12SINUM,LiZhi2022}
construct polynomials from linear combinations of partial
derivatives of different orders of the original equations.
With exact coefficients, these conditions vanish at $z^\ast$.
In contrast, the algorithm in Theorem~\ref{th:gap1:main} uses
only evaluations of $h$ and its Jacobian.


Many algorithms have also been developed that solve system $f(z)=0$ without introducing a homotopy.
Some of them use only evaluations of function $f$ and its Jacobian
and achieve superlinear convergence for some classes of singular systems~\cite{DeckerKelley1982,Kelley1986,Mannel2023}. 
However, these papers establish superlinear convergence only under some additional nondegeneracy assumptions involving higher-order derivatives of $f$ at the singular root.

\subsection{Estimating ratio $k_1/c$ of the Puiseux series}
Most endgames in Numerical Algebraic Geometry require estimating the ratio $k_1/c$ in the Puiseux series from existing tracked points.
Existing techniques include the trial-and-error method (``\cSORTcode''), which evaluates prediction errors across 
a range of candidate integer values for $c$ \cite{MSW1992, SommeseWampler} using two or three input points.
Another established technique is the geometric sequence sampling approach (``\cLOGcode'')  
which isolates the leading fractional exponent by analyzing the logarithmic differences of path samples taken at geometrically decreasing 
parameter values $t,\rho t,\rho^2t$ \cite{bates2011parallel, SommeseWampler}. 
Note that \cSORTcode\ selects the best value from a prescribed candidate set, while \cLOGcode\ directly estimates $k_1/c$: it converges to $k_1/c$ as $t\rightarrow 0$, assuming that the input points are exactly on the homotopy.

We introduce a new method for estimating $k_1/c$  (for systems of arbitrary corank $\kappa$) that we call ``\cRATIOcode''. This is the method that we use in the proof of Theorem~\ref{th:gap1:main}. Similar to \cLOGcode, it directly estimates $k_1/c$, but uses only two points and thus does not require geometric sampling.
(Unlike \cLOGcode, \cRATIOcode\ also needs derivatives $\dot z$, but these derivatives need to be computed anyway during the tracking). 
Our experiments indicate that \cRATIOcode\  exhibits fewer oscillations than \cSORTcode\ and \cLOGcode.

Note that $k_1/c$ can also be estimated via the Cauchy integral method \cite{bates2011parallel, SommeseWampler}, assuming that $k_1=1$. Unlike the local estimators above, it requires tracking the path around a sufficiently small loop enclosing $t=0$ until the solution returns to its initial value.

\section{Background and notation}

For a function $F:\mathbb C^n\rightarrow \mathbb C^m$
and variables $x=(u,v)$ the Jacobian of $F$ with respect to $u$ is denoted either as $F_u(x)$
or as $D_u F(x)$. Both $x$ and $F(x)$ are treated as column vectors.

The Hermitian transpose of matrix $A$ is denoted as $A^\dag$.

Notations $z$ and $z(\cdot)$ will denote different objects: $z(\cdot)$ is a function, while $z$ is a specific value which is not necessarily related to $z(\cdot)$.

Throughout this paper, for a point $(z,t)\in\mathbb C^{n+1}$ we denote 
\begin{equation}
\label{eq:dot-z}
	\dot{z}= - h_z (z,t)^{-1} h_t(z,t)
\end{equation}
This definition depends also on $t$; the value of $t$ should always be clear from the context.
Note, if $z=\bar z(t)$ for a differentiable function $\bar z(\cdot)$ satisfying
$h(\bar z(t),t)=0$ in some neighborhood of $t$ then $\dot z=\frac d {dt} \bar z(t)$.

We define variety $\calV\subseteq\mathbb C^{n+1}$ as $\calV=h^{-1}(0)$.

If $t_1,t_2$ are complex values in $\mathbb C$ then $[t_1,t_2]$ denotes the interval $\{\alpha t_1 + (1-\alpha) t_2\::\:\alpha\in[0,1]\subset\mathbb R\}$.

Notation $f(t)=O(g(t))$ for complex-valued functions $f,g$ will mean that $\|f(t)\|\le C \|g(t)\|$ if $|t|\le t_{\max}$,
for some constants $C>0$ and $t_{\max}>0$.

\subsection{Power-series endgame}\label{section:Background}
One classical approach to tracking the path close to the root is the {\em power-series endgame}. It
maintains a set of pairs of the form $$\calX=\{x_i=(z_i,t_i)\}_{i=0,1,\ldots}$$ where $t_i\in(0,1]$ and $z_i$ approximates $z(t_i)$.
At each step it does the following.
 
\begin{itemize}
\item Using pairs in $\calX$,
estimate the first $\ell+1$ coefficients of series~\eqref{eq:Puiseux} together with ratios $k_i/c$, obtaining approximation
\begin{equation}\label{eq:Puiseux:approx}
\hat z(t)= \hat z^\ast + \hat a_{k_1} t^{k_1/c} + \ldots + \hat a_{k_\ell} t^{k_\ell/c} .
\end{equation}
\item Select ``target'' value $\hat t\in(0,1]$. Usually one takes $\hat t =\rho\cdot t_{\circ}$ where $t_{\circ}$ is the smallest value present in $\calX$,
and parameter $\rho\in(0,1)$ is either fixed or updated adaptively based on the success / failure status of previous steps.
\item Predictor step: compute vector $\hat z=\hat z(\hat t)$.
\item Corrector step: compute $z$ by applying several steps of  Newton's method to solve system $h(z,\hat t)=0$ using $\hat z$ as the starting point.
If Newton's method converges according to a certain criterion
then add pair $(z,\hat t)$ to $\calX$.
\end{itemize}

Popular choices for the predictor are a {\em linear} and a {\em cubic} predictors;
they estimate respectively the first two and the first four terms in~\eqref{eq:Puiseux:approx}.
We refer to Section~\ref{section:PredBounds} for a further discussion of predictors.


\subsection{Newton's method}
In this section we state the classical Kantorovich theorem about convergence of  Newton's method
which we will need later
\cite{FerreiraSvaiter}.

\begin{theorem}\label{th:NK}

    Let $X$ and $Y$ be Banach spaces, $\Omega$ be a subset of $X$ and $F$ be a continuous non-linear operator, 
    $F: \Omega \mapsto Y$, such that $F$ is continuously Fr\'{e}chet-differentiable on $\inter \left( \Omega \right)$. 
    For an initial guess $x_0 \in \Omega$ and for positive reals $L, C \in \mathbb{R_+}$ assume that
    \begin{itemize}
        
        \item $F'(x_0)$ is non-singular;

        \item $\left\| \left[ F'(x_0) \right]^{-1} \Big( F'(x) - F'(y) \Big) \right\| \ \leq \ L \| x - y \| \qquad \forall x, y \in \Omega$;

        \item $\left\| \left[ F'(x_0) \right]^{-1} F(x_0) \right\| \ \leq \ C$;

        \item $2 C L \ < \ 1$.
        
    \end{itemize}

    Consider $r\in(r_-,r_+]$ where
    \[
    r_- \ = \ \frac{1 - \sqrt{1 - 2 C L}}{L}, \qquad r_+ \ = \ \frac{1 + \sqrt{1 - 2 C L}}{L}.
    \]
    If $B(x_0, r) = \{x\in X\::\|x-x_0\| < r\:\} \subset \Omega$ then the sequence $\{ x_k \}$ generated by Newton's method for solving non-linear equation $F(x) = 0$ with initial point $x_0$,
    \[
    x_{k+1} \ = \ x_k - \left[ F'(x_k) \right]^{-1} F(x_k) \qquad \forall k \geq 0,
    \]
    is contained in $B(x_0, r)$, converges to the unique zero $x^* \in B(x_0, r)$ of $F$ and the following error bound holds:
    \[
    \| x_{k+1} - x^* \| \ \leq \ \frac{L}{2 \sqrt{1 - 2 C L}} \| x_k - x^* \|^2 \qquad \forall k \geq 0
    \]
    
\end{theorem}

\subsection{Properties of corank-1 systems}\label{sec:corank1-properties}
The following facts are consequences of classical local analytic geometry~\cite{MSW1991,MSW1992}.
For completeness, we give a proof in Appendix~\ref{sec:prop:k1:proof}.
\begin{proposition}\label{prop:k1}
Suppose that $\kappa=1$ and the conditions \textup{(a)}-\textup{(d)} of Assumption~\ref{assumption:main}
are satisfied. Then condition \textup{(e)} is also satisfied, and the following holds:
if $z(t)=\sum_{i=0} a_i t^{i/c}$ is a Puiseux series in $\Pi$ then $a_1\ne 0$ (hence $k_1=1$) and
Puiseux series in $\Pi$ are parameterized by complex $c$-roots $\omega$ of $1$:
$$
\Pi=\left\{z(t)=\sum_{i=0}^\infty a_i\omega^i t^{i/c}\::\:\omega^c = 1\right\}
$$
In particular, $|\Pi|=c$.
\end{proposition}


\section{Linear predictor}\label{sec:linear-predictor}
To prove Theorem~\ref{th:gap1:main}, we will use a linear predictor that estimates $z^\ast$, $a_{k_1}$ and $k_1/c$.
In this section we will analyze the accuracy of this predictor, assuming in particular that the input points
satisfy the homotopy only approximately. 
Note, here we do {\bf not} assume that $\kappa=1$.

Our predictor will depend on parameters~$\gamma,\beta,k_1^{\max}$; these are positive constants that will be specified later.
Given input point $x=(z,t)$, it does the following.
\begin{enumerate}
\item Set $t_1=(1-|t|^{\gamma}) \cdot t$.
\item Run  Newton's method to solve the system $h(z,t_1)=0$ starting with a point $$z_0=z+\dot z (t_1-t),$$
until getting a point $z_1$ with $\|h(z_1,t_1)\|\le |t_1|^\beta$.
\item Find positive integers $c,k_1$ with $k_1\le k_1^{\max}$ that minimize $$\left| \frac{\|t\dot z-t_1\dot z_1\|}{\|z-z_1\|} - \frac  {k_1}c\right|.$$
\item Output predictor $\hat z=z - \tfrac c {k_1}  \dot z t$.
\end{enumerate}

We will prove the following result.
\begin{theorem}\label{th:linear-predictor}
There exist constants $\gamma_{\min}>0$, $\eta>0$ with the following property.
Suppose that $\gamma>\gamma_{\min}$, $\beta>\eta\gamma$, and index $k_1$ of each Puiseux series $z(\cdot)\in\Pi$ satisfies $k_1\le k_1^{\max}$.
Then there exists a neighborhood $\Omega$ of $x^\ast$ such that
any $x = (z, t)\in\Omega$ with $\| h(x) \| \le |t|^\beta$ satisfies the following. \\
\textup{ (i)}  $\| \hat z - z^* \| = O\left( \| z - z^* \|^{k_2/k_1} \right)$
where $k_1,k_2$ are the indices in eq.~\eqref{eq:Puiseux} of the Puiseux series $z(\cdot)\in\Pi$ with the smallest ratio $k_2/k_1$. \\
\textup{ (ii)}  Newton's method in step 2 terminates after $O(\log (1+\beta))$ iterations. \\
If $\kappa=1$ then  $\gamma_{\min}<1$ and $k_{1}=1$, and so one can take $\gamma=1$ and $k_{1}^{\max}=1$.
\end{theorem}

The remainder of this section is devoted to the proof of this theorem. In these proofs we will often omit
the phrase ``there exists a neighborhood $\Omega$ of $x^\ast$ such that ...'', making it implicit.
For example, we will write $O(|t|^a)\le |t|^b$ when $a>b$; this would hold if $|t|$ is sufficiently small.
Also, in each lemma we will implicitly assume that the current $\Omega$ is contained in the neighborhoods considered in all previous statements.
One of them is the neighborhood $\Omega_z\times\Omega_t$ defined in Assumption~\ref{assumption:main}(e),
so all Puiseux series $z(\cdot)\in\Pi$ will be assumed to be convergent in the considered neighborhood.

First, we analyze
what happens when the points lie exactly on the curve. 
\begin{lemma}\label{thm:BoundLinPred}
Consider Puiseux series $z(\cdot)\in\Pi$ associated with integers $c,k_1,k_2$.
There exists a neighborhood $\Omega$ of $x^\ast$ such that points $x=(z,t)=(z(t),t)\in\Omega$ satisfy the following.
\\
\textup{ (a)}
$
\lim\limits_{t\rightarrow 0}\sup\limits_{t_1\in[0,t)} \;\; \left| \frac{\|t\dot z(t)-t_1\dot z(t_1)\|}{\|z(t)-z(t_1)\|} - \frac  {k_1}c\right| = 0
$.
\\
\textup{ (b)} $\|\hat z-z^\ast\|=O(\|z-z^\ast\|^{k_2/k_1})$ assuming that $\hat z$ was computed with the correct value of $k_1/c$.
\end{lemma}

\begin{proof}
\myparagraph{Part (a)}
Plugging the Puiseux series into $z(\cdot)$ for coordinate $i\in[n]$ yields
\begin{align*}
z(t)[i] - z(t_1)[i] &= a_{k_1}[i](t^{k_1/c} - t_1^{k_1/c}) + O(t^{k_2/c} - t_1^{k_2/c}) = a_{k_1}[i] \cdot \tau + O(\sigma) \\
t\dot z(t)[i] - t_1\dot z(t_1)[i] &= \frac{k_1}{c}\cdot a_{k_1}[i](t^{k_1/c} - t_1^{k_1/c}) + O(t^{k_2/c} - t_1^{k_2/c}) = \frac{k_1}{c}\cdot  a_{k_1}[i] \cdot \tau + O(\sigma)
\end{align*}
where we denoted $\tau = t^{k_1/c} - t_1^{k_1/c}$, $\sigma=t^{k_2/c} - t_1^{k_2/c}$.
Therefore,
\begin{align*}
\|z(t) - z(t_1)\| &=  \|a_{k_1}\| \cdot |\tau| + O(\sigma) \\
\|t\dot z(t) - t_1\dot z(t_1)\| &= \frac{k_1}{c}\cdot  \|a_{k_1}\| \cdot |\tau| + O(\sigma)
\end{align*}

Note that 
$
\frac \sigma \tau= t^{(k_2-k_1)/c} \cdot \frac {1-(t_1/t)^{k_2/c}}{1-(t_1/t)^{k_1/c}}
$
and hence $\left|\frac \sigma \tau\right|=O(|t|^{(k_2-k_1)/c})$
(since $\frac{1-\rho^{k_2/c}}{1-\rho^{k_1/c}}\le \frac {k_2}{k_1}$ for $\rho\in[0,1)$).
This implies the claim.

\myparagraph{Part (b)}
Assuming we have successfully extracted the exact leading exponent ratio ${k_1}/{c}$, the target prediction $\hat{z}$ (aiming for $t=0$) is computed via the linear ideal predictor:
$
\hat{z} = z(t) - \frac{c}{k_1} t\dot{z}(t)
$.
Substituting the series expansions into this predictor equation we have
\begin{align*}
\hat{z} &= z^* + \sum_{j=k_{1}}^{\infty} a_{j} t^{j/c} -\frac{c}{k_1}\sum_{j=k_{1}}^{\infty} \frac{j}{c} a_{j} t^{j/c}  
= z^* + \left(1 - \frac{k_2}{k_1}\right) a_{k_2} t^{k_2/c} + {O}(t^{k_3/c})
\end{align*}

The leading terms $a_{k_1} t^{k_1/c}$ cancel exactly. Therefore, isolating the error gives
$
\|\hat{z} - z^*\| = {O}(|t|^{k_2/c})
$ as $t \to 0$.
Since $\|z(t) - z^*\| = \|a_{k_1}\| \cdot |t|^{k_1/c} + {O}(|t|^{k_2/c})$, we can asymptotically bound the parameter $t$ as
$
|t|^{1/c} = {O}\left(\|z(t) - z^*\|^{1/k_1}\right)
$.
Substituting this relation back into our predictor error bound produces the final geometric bound
\[
\|\hat{z} - z^*\| = {O}\left( \left( \|z(t) - z^*\|^{1/k_1} \right)^{k_2} \right) = {O}\left( \|z(t) - z^*\|^{k_2/k_1} \right)
\]

\end{proof}



Next, we analyze the existence and behaviour of matrix $h^{-1}_z(z(t),t)$ in a neighborhood of $0$.
\begin{lemma}\label{lemma:hz-bound}
Consider formal series $z(\cdot)\in\Pi$ associated with integers $k_1,c$.
There exists a constant $\delta>0$ and 
a punctured neighborhood $\Omega$ of $x^\ast$
such that for any $z(\cdot)\in\Pi$ and $x=(z,t)=(z(t),t)\in\Omega$
matrix $h_z(x)$ is nonsingular, and
there holds
 $\|h^{-1}_z(x)\|\le |t|^{-\delta}$, $\|\dot z(t)\|=O(|t|^{k_1/c-1})$ and $\|\ddot z(t)\|=O( |t|^{k_1/c-2})$.
 Furthermore, there exists component $i\in[n]$ such that $|\dot z_i(t)|=\Theta(|t|^{k_1/c-1})$.
 
If $\kappa=1$ then $\delta\in(0,1)$ and $k_1=1$.
\end{lemma}
\begin{proof}

Define function $\tilde z(s)=z(s^c)$.
This is an analytic function at $0$, as it is given by a convergent power series at some neighborhood of $0$.
Also, $z(t)=\tilde z(t^{1/c})$ for some branch $t\mapsto t^{1/c}$.

Define $J(t)=h_z(z(t),t)$ and $\tilde J(s)=J(s^c)$. 
Note that the entries of matrix $\tilde J$ are analytic functions of $s$ since $\tilde J(s)=h_z(\tilde z(s),s^c)$,
and thus $\det \tilde J(s)$ is also an analytic function of $s$.
By Assumption~\ref{assumption:main}, $\det \tilde J(s)$ is not identically zero,
therefore $\det \tilde J(s)=s^d \varphi(s)$ for some
integer $d\ge 1$ and analytic function $\varphi(s)$ with $\varphi(0)\ne 0$.
In particular, $\det \tilde J(s)\ne 0$ in some punctured neighborhood of $0$.
By Cramer's rule, $\tilde J(s)^{-1}=\frac{{\texttt{adj}} \tilde J(s)}{\det \tilde J(s)}=\frac{{\texttt{adj}} \tilde J(s)}{\varphi(s)}\cdot s^{-d}$.
Function $\frac{{\texttt{adj}} \tilde J(s)}{\varphi(s)}$ is analytic at $0$, therefore
$\|\tilde J(s)^{-1}\|=O(|s|^{-d})$ and hence $\|h_z(z(t),t)^{-1}\|=\|\tilde J(t^{1/c})^{-1}\|=O(|t|^{-d/c})<|t|^{-\delta}$
for any fixed $\delta > d/c$.

By differentiating the formal series~\eqref{eq:Puiseux} we obtain $\dot z(t)=\tfrac{k_1}c a_{k_1} t^{k_1/c-1} (1+o(1))$ and
$\ddot z(t)=\tfrac{k_1}c(\tfrac{k_1}c-1) a_{k_1} t^{k_1/c-2} (1+o(1))$.
This implies that $\|\dot z(t)\|=O(|t|^{k_1/c-1})$ and $\|\ddot z(t)\|=O( |t|^{k_1/c-2})$,
and also $|\dot z_i(t)|=\Theta(|t|^{k_1/c-1})$ for all components $i\in[n]$ with $(a_{k_1})_i\ne 0$.

Let us now assume that $\kappa=1$. By Proposition~\ref{prop:k1}, we have $k_1=1$.
Let $J^\ast=U^\ast \Sigma^\ast (V^\ast)^\dag$ and $h_z= U \Sigma V^\dag$ be \glspl{SVD}  of $J^\ast=h_z(x^\ast)$ and
$h_z(x)$ respectively, with 
$\Sigma^\ast={\texttt{diag}}(\sigma^\ast_1,\ldots,\sigma^\ast_n)$, $\sigma^\ast_1\ge\ldots\ge \sigma^\ast_{n-1}>\sigma^\ast_n=0$,
$\Sigma={\texttt{diag}}(\sigma_1,\ldots,\sigma_n)$, $\sigma_1\ge \ldots\ge \sigma_n\ge 0$.
Let $\{u^\ast_i\},\{v^\ast_i\},\{u_i\},\{v_i\}$ be the columns of $U^\ast,V^\ast,U,V$ respectively.
	Vector $u^\ast_n$ is the left singular vector of $J^\ast$ for value $\sigma^\ast_n=0$ (i.e. $(u_n^\ast)^\dag J^\ast=0$);
by Assumption~\ref{assumption:main}(d), we have $(u^\ast_n)^\dag h^\ast_t \ne 0$.

It follows from Wedin's theorem \cite[Theorem 4.1]{MatrixPerturbationTheory} that $\sigma_n,u_n,v_n$ depend continuously on matrix~$h_z$
(as long as singular value $\sigma_n$ has multiplicity 1).
Therefore, there exists a neighborhood of $x^\ast$ in which points $x\in \calV$ satisfy
$ \sigma_1 \ge \ldots \ge \sigma_{n-1}\ge \Theta(1)$, $\|h_t\|= \Theta(1)$ and 
 $|u_n^\dag h_t|= \Theta(1)$.
For points $x=(z,t)\ne x^\ast$ in this neighborhood we have
\begin{equation*}
  \begin{aligned}
    \dot z 
    &= -h_z^{-1} h_t = -(V \Sigma^{-1} U^\dag) h_t 
    =
     -\sum_{i=1}^n \sigma_i^{-1} v_i u_i^\dag h_t 
    = \left(-\sum_{i=1}^{n-1} \sigma_i^{-1} v_i u_i^\dag h_t\right) - \left(\sigma_n^{-1} v_n u_n^\dag h_t\right)
  \end{aligned}
\end{equation*}
and so
$$
\sigma_n^{-1} v_n u_n^\dag h_t = -\dot z - \sum_{i=1}^{n-1} \sigma_i^{-1} v_i u_i^\dag h_t
$$
Taking norms gives
$$
\sigma_n^{-1} \cdot\|v_n\| \cdot |u_n^\dag h_t| \le \|\dot z\| + \sum_{i=1}^{n-1} \sigma_i^{-1} \|v_i\|  \cdot \|u_i\|\cdot \| h_t\|
$$
We have $\|v_i\|=\|u_i\|=1$, and so $\|h_z^{-1}\|=\sigma_n^{-1}\le \frac{1}{\Theta(1)}(O(|t|^{k_1/c-1}) + O(1))=O(|t|^{k_1/c-1})$.
Thus, any constant $\delta>1-k_1/c$ will satisfy the claim of the lemma, so we can indeed choose $\delta<1$.
\end{proof}
Let us fix constants $\delta>0,\Delta<1,\Lambda>-1$ such
that for each formal series $z(\cdot)\in\Pi$
we have $\delta>\delta^{z(\cdot)}$, $\Delta>1-\tfrac{k_1^{z(\cdot)}}{c^{z(\cdot)}}$,
$\Lambda>\tfrac{k_1^{z(\cdot)}}{c^{z(\cdot)}}-1$ (here the superscript $z(\cdot)$ denotes the value associated with formal series $z(\cdot)$, and value $\delta^{z(\cdot)}$ comes from Lemma~\ref{lemma:hz-bound}.)
By Lemma~\ref{lemma:hz-bound}, the following holds for all $x=(z,t)\in\calV$ in some punctured neighborhood of $x^\ast$:
\begin{subequations} \label{X}
\begin{eqnarray}
\|h^{-1}_z(x)\| & \le & |t|^{-\delta} \label{Xa} \\
\|\dot z\| & \le & |t|^{-\Delta} \label{Xb} \\
\|\ddot z\| & \le & |t|^{-\Delta-1} \label{Xc} 
\end{eqnarray}
\end{subequations}
Note, if $\kappa=1$ then we can have $\delta<1$ and $\Lambda<0$.

We define $\gamma_{\min}=\delta+ \tfrac{\Delta-1}2$. Note, if $\kappa=1$ then $\gamma_{\min}<1$.
We thus assume from now on that
\begin{equation}\label{eq:gamma-choice}
\gamma > \delta+ \tfrac{\Delta-1}2
\end{equation} 

\begin{lemma}\label{L_0}
For any constant $\alpha>0$ there exist another constant $\beta>0$ and neighborhood $\Omega$ of $x^\ast$ 
satisfying the following:
if $x=(z,t)\in\Omega$ and $\|h(x)\|\le |t|^\beta$
then 
there exists  Puiseux series $\bar z(\cdot)\in\Pi$ such that
$\|z-\bar z\|\le |t|^{\alpha}$ and $\|\dot z-\dot{\bar z}\|\le|t|^{\alpha-2\delta}$ where $\bar z=\bar z(t)$.
\end{lemma}
\begin{proof}
We can assume w.l.o.g.\ that $\alpha\ge \delta$ and $\alpha+\delta>1$ (by increasing $\alpha$, if necessary; this will not affect the claim).
By the classical {\L}ojasiewicz inequality \cite{Lojasiewicz1959},
there exist constants $C>0$, $\theta>0$ such that every $x$ in some neighborhood of $x^\ast$ satisfies
\begin{equation}\label{eq:Lojasiewicz}
{\texttt{dist}}(x,\calV)\le C\|h(x)\|^\theta
\end{equation}
We will show the lemma for any value $\beta$ satisfying
\begin{equation}\label{ADGHAKJDSFAK}
\beta > \tfrac {\alpha+\delta} \theta
\end{equation}

Consider $x=(z,t)$ with $\|h(x)\|\le |t|^\beta$ in some neighborhood of $x^\ast$. 
By~\eqref{eq:Lojasiewicz}, there exists $x'=(z',t')\in\calV$ with $\|x-x'\|\le C\|h(x)\|^\theta\le C(|t|^\beta)^\theta=C|t|^{\bar\alpha+\delta}$
where $\bar\alpha>\alpha$ is a constant.
By shrinking the neighborhood if necessary, we can assume that $C|t|^{\bar \alpha+\delta}<\frac 12|t|$.
We have $z'=\bar z(t')$ for some Puiseux series $\bar z(\cdot)\in \Pi$.
By the mean value inequality, $\|\bar z(t)-\bar z(t')\|\le \|\dot {\bar z}(\tau)\| \cdot |t-t'|$ for some $\tau\in[t,t']$.
Since  $|t-t'|\le \|x-x'\|\le C|t|^{\bar \alpha+\delta}<\tfrac 12 |t|$,
we must have $|\tau|=\Theta(|t|)$. 
Denoting $\bar z=\bar z(t)$,  we get
$$
\|z-\bar z\|\le \|z-z'\| + \|z'-\bar z\| \le \|x-x'\| + \|\dot {\bar z}(\tau)\|\cdot |t-t'| \le \|x-x'\| \cdot (1 + \|\dot {\bar z}(\tau)\|)
$$
since $\max\{\|z-z'\|,|t-t'|\}\le \|x-x'\|$.
Using eq.~\eqref{Xa}, we get  $\|\dot {\bar z}(\tau)\|=\| h_z(\bar z(\tau),\tau)^{-1}  h_t(\bar z(\tau),\tau) \|\le \| h_z(\bar z(\tau),\tau)^{-1} \|\cdot \| h_t(\bar z(\tau),\tau) \| \le |\tau|^{-\delta}\cdot O(1)=O(|t|^{-\delta})$.
This yields
$$
\|z-\bar z\| \le C|t|^{\bar \alpha+\delta} \cdot (1 + O(|t|^{-\delta}))=O(|t|^{\bar \alpha})
$$
Since $\bar\alpha>\alpha$, taking a sufficiently small neighborhood will ensure that the last expression is at most $|t|^\alpha$.

To prove the bound on $\|\dot z-\dot{\bar z}\|$, we will use the following fact:
\begin{itemize}
\item {\em Suppose that $A,B\in\mathbb C^{n\times n}$, $a,b\in\mathbb C^{n\times 1}$,  $A$ is invertible and $\|A^{-1}\| \; \| A-B\|< 1$. Then
\begin{equation}
\|A^{-1}a-B^{-1}b\|\le \|A^{-1}\| \; \|a-b\|+\frac{\|A^{-1}\|^2\; \|A-B\|}{1-\|A^{-1}\| \; \|A-B\|}\|b\| \label{eq:GALSKHFDALGAS}
\end{equation}
Indeed, the assumption implies that $B$ is invertible and $\|B^{-1}\|\le \frac{\|A^{-1}\|}{1-\|A^{-1}\|\;\|A-B\|}$. We have $A^{-1}a-B^{-1}b=A^{-1}(a-b)+A^{-1}(B-A)B^{-1}b$ and hence
$\|A^{-1}a-B^{-1}b\|\le \|A^{-1}\|\;\|a-b\|+\|A^{-1}\|\; \|B-A\| \; \|B^{-1}\| \; \|b\|$, which yields~\eqref{eq:GALSKHFDALGAS}.
}
\end{itemize}
Let us plug $A=h_z(\bar z,t)$, $a=h_t(\bar z,t)$, $B=h_z(z,t)$, $b=h_t(z,t)$.
Since $h$ is analytic, we have $\|A-B\|\le O(\|z-\bar z\|)\le O(|t|^{\bar \alpha})$,
 $\|a-b\|\le O(\|z-\bar z\|)\le O(|t|^{\bar \alpha})$ and $\|b\|\le O(1)$.
By eq.~\eqref{Xa},  $\|A^{-1}\|\le |t|^{-\delta}$.
Plugging this into~\eqref{eq:GALSKHFDALGAS} gives
$$
\|\dot z-\dot{\bar z}\| 
= \|A^{-1}a-B^{-1}b\|
\le |t|^{-\delta} \cdot O(|t|^{\bar \alpha}) + \frac{|t|^{-2\delta}\cdot O(|t|^{\bar \alpha})}{1-|t|^{-\delta}\cdot O(|t|^{\bar \alpha})}\cdot O(1)
\le O(|t|^{\bar \alpha-2\delta})
$$
since $\bar \alpha>\delta$.
Since $\bar\alpha>\alpha$, taking a sufficiently small neighborhood will ensure that the last expression is at most $|t|^{\alpha-2\delta}$.
\end{proof}

We now proceed with the proof of Theorem~\ref{th:linear-predictor}.
Fix $\alpha$ that satisfies
\begin{equation}\label{eq:alpha-choice}
\alpha\;\;>\;\;\max\left\{\;\; \Lambda+\gamma+\max\{ 2\delta, 1 \} \;\;,\;\;\tfrac {k_2}c\;\;,\;\; 2\delta-1+\tfrac{k_2}{c}\;\; ,\;\; 2\gamma-\Delta+1 \;\; \right\}
\end{equation}
for coefficients $c,k_1,k_2$ of all Puiseux series $z(\cdot)\in\Pi$.
Let $\beta>0$ be the constant specified in Lemma~\ref{L_0} for this value of $\alpha$.
Note that this value can be chosen so that $\beta=\Theta(\gamma)$ (see eq.~\eqref{ADGHAKJDSFAK}).
Assume that the input point $x=(z,t)$ in the appropriate neighborhood of $x^\ast$ satisfies $\|h(x)\|\le|t|^\beta$.
By Lemma~\ref{L_0}, there exists Puiseux series $\bar z(\cdot)\in\Pi$ such that $\bar z=\bar z(t)$ satisfies
\begin{subequations}\label{eq:z-minus-bar-z}
\begin{eqnarray}
\|z-\bar z\|&\le& |t|^\alpha \\
\|\dot z-\dot{\bar z}\|&\le& |t|^{\alpha-2\delta}
\end{eqnarray}
\end{subequations}

Let us denote $\bar z_1=\bar z(t_1)$. Our next goal will be to show that point $z_1$ constructed by the algorithm
is close to $\bar z_1$, and gives the correct value of the ratio $k_1/c$. Recall that $z_1$ is obtained via  Newton's method for system $F(z)=0$ where $F(z)=h(z,t_1)$.

\begin{lemma}
\textup{ (a)} The starting point $z_0=z+\dot z (t_1-t)$ satisfies $\|z_0-\bar z_1\|\le  O(|t|^{2\gamma-\Delta+1})$.
 \\
\textup{ (b)} The preconditions of Theorem~\ref{th:NK} hold with $L=\Theta(|t|^{-\delta})$, $C=\tfrac 4{9L}$.
\\
\textup{ (c)} Newton's method terminates after $O(\log (1+\beta))$ iterations.
\\
\textup{ (d)} It produces point $z_1$ satisfying $\|z_1-\bar z_1\|\le |t_1|^{\alpha}$ and $\|\dot z_1-\dot{\bar z}_1\|\le |t_1|^{\alpha-2\delta}$.
\\
\textup{ (e)} There exists a neighborhood $\Omega$ of $x^\ast$ with the following
property: if $x\in\Omega$ then the method produces the correct ratio $k_1/c$ for the Puiseux series $\bar z(\cdot)$.

\end{lemma}
\begin{proof}
\myparagraph{Part (a)}
By a Taylor formula with an integral remainder,
$$
z_0-\bar z_1=z_0-\left(\bar z+(t_1-t)\dot {\bar z} + R\right)=(z-\bar z)+(t_1-t)(\dot z- \dot {\bar z})-R,
\quad
R=(t_1-t)^2\int_0^1 (1-s)\ddot{\bar z}(t+s(t_1-t))ds
$$
We have $|t_1-t|=|t|^{\gamma+1}$ and $\|\ddot{\bar z}(\tau)\|\le |\tau|^{-\Delta-1}$ for each $\tau\in[t,t_1]$, and so
$$
\|z_0-\bar z_1\|\le|t|^\alpha + |t|^{\gamma+1}\cdot |t|^{\alpha-2\delta} + \tfrac 12 |t|^{2(\gamma+1)}\cdot |\tau|^{-\Delta-1}
=O(|t|^{2\gamma-\Delta+1})
$$
since  $\alpha\ge 2\gamma-\Delta+1$ and $\gamma+1+\alpha-2\delta\ge 2\gamma-\Delta+1$.

\myparagraph{Part (b)}
Since $F_z(z)=h_z(z,t_1)$ is analytic in a neighborhood of $z^\ast$, we have
$\|F_z(z')-F_z(z'')\|\le O(\|z'-z''\|)$ when $z',z''$ are in a certain neighborhood of $z^\ast$.
In particular, we have $\|F_z(z_0)-F_z(\bar z_1)\|\le O(\|z_0-\bar z_1\|)\le O(|t|^{2\gamma-\Delta+1})$.
Also, $\|F^{-1}_z(\bar z_1)\|\le O(|t_1|^{-\delta})=O(|t|^{-\delta})$ by eq.~\eqref{Xa}.
Therefore, $\|F_z(z_0)-F_z(\bar z_1)\|\;\|F^{-1}_z(\bar z_1)\|\le \tfrac 12$ if $|t|$ is sufficiently small
(since $2\gamma-\Delta+1-\delta>0$ by eq.~\eqref{eq:gamma-choice}). This implies that
$$
\|F^{-1}_z(z_0)\| \le \frac{\|F^{-1}_z(\bar z_1)\| }{1-\|F_z(z_0)-F_z(\bar z_1)\|\;\|F^{-1}_z(\bar z_1)\|}\le O(|t|^{-\delta})
$$
We conclude that $\|F^{-1}_z(z_0)(F_z(z')-F_z(z''))\|\le O(|t|^{-\delta})\cdot O(\|z'-z''\|)$ when $z',z''$ are in a certain neighborhood of $z^\ast$,
and hence the first two preconditions of Theorem~\ref{th:NK} hold with $L=\Theta(|t|^{-\delta})$.
We have $\| F(z_0) \|=\| F(z_0) - F(\bar z_1)\|\le O(\|z_0-\bar z_1\|)\le O(|t|^{2\gamma-\Delta+1})$,
and so the third precondition holds with $C\ge C_0=O(|t|^{-\delta})\cdot O(|t|^{2\gamma-\Delta+1})$.
We obtain $C_0L=O(|t|^{2\gamma-2\delta-\Delta+1})=O(|t|^\varepsilon)$ where $\varepsilon>0$ by the choice of $\gamma$.
Therefore, we can indeed set $C=\tfrac 4{9L}$ if $|t|$ is sufficiently small.

\myparagraph{Parts (c,d)}
Theorem~\ref{th:NK} yields that equation $F(z)=0$ has a unique solution in the ball $B(z_0,r)$ 
for any $r\in(r_-,r_+]=(\tfrac 2{3L},\tfrac 4{3L}]$.
We have $F(\bar z_1)=0$ and $\|z_0-\bar z_1\|\le O(|t|^{2\gamma-\Delta+1})<r_-$ if $|t|$ is sufficiently small
(since $2\gamma-\Delta+1>\delta$ by eq.~\eqref{eq:gamma-choice}), so this unique solution must be $\bar z_1$.

Let us denote $w_k=z^{(k)}-\bar z_1$ where $z^{(k)}$ is the iterate at step $k$ (with $z^{(0)}=z_0$).
Theorem~\ref{th:NK} gives that $\|w_{k+1}\|\le \bar C \cdot |t|^{-\delta} \|w_{k}\|^2$.
We have $\|w_0\|\le O(|t|^\lambda)$ where $\lambda = 2\gamma-\Delta+1 > 2\delta$,
therefore $\|w_k\|\le |t|^{\Theta(\delta 2^k)}$.
Since $h$ is an analytic function, we have $\|h(z^{(k)},t_1)\|=\|h(z^{(k)},t_1)-h(\bar z_1,t_1)\|\le O(\|z^{(k)}-\bar z_1\|)\le |t|^{\Theta(\delta 2^k)}$.
We conclude that for any fixed $\beta$ we will have $\|h(z^{(k)},t_1)\|\le |t_1|^\beta$ after $k=O(\log (1+\beta))$ iterations.

By Lemma~\ref{L_0}, there exists $\bar{\bar z}(\cdot)\in\Pi$ with $\|z_1-\bar{\bar z}_1\|\le |t_1|^\alpha$ and $\|\dot z_1-\dot {\bar{\bar z}}_1\|\le |t_1|^{\alpha-2\delta}$
where $\bar{\bar z}=\bar {\bar z}(t_1)$.
Since $\alpha>\delta$, we can assume that $|t_1|^\alpha < r_-=\Theta(|t|^\delta)$.
This implies that $\|z_0-\bar{\bar z}_1\|\le \|z_0-z_1\| + \|z_1-\bar{\bar z}_1\| \le r_-+ |t_1|^\alpha  <r_- + r_-=r_+$.
We have $F(\bar z_1)=F(\bar{\bar z}_1)=0$ and $\bar z_1,\bar{\bar z}_1\in B(z_0,r_+)$, therefore
 $\bar z_1=\bar{\bar z}_1$ and hence $\dot{\bar z}_1=\dot{\bar{\bar z}}_1$.

\myparagraph{Part (e)} 
%
%
Denote $A=\lVert t\dot z-t_1\dot z_1\rVert$, $B=\lVert z-z_1\rVert$,
and $R=k_1/c$. Choose $\varepsilon>0$ such that every other candidate
ratio $p/q$, with positive integers $p\le k_1^{\max}$ and $q$, has
$|p/q-R|>3\varepsilon$. Such a choice exists because bounded-numerator
rationals have no positive accumulation point. It suffices to show
$|A/B-R|<\varepsilon$ in a sufficiently small neighborhood.

Let $\bar A=\lVert t\dot{\bar z}(t)-t_1\dot{\bar z}(t_1)\rVert$
and $\bar B=\lVert\bar z(t)-\bar z(t_1)\rVert$. Differentiating the
convergent Puiseux expansion and integrating along the radial segment
from $t_1=(1-|t|^\gamma)t$ to $t$ gives
$$
 \bar z(t)-\bar z(t_1)
 =\int_{t_1}^{t}\dot{\bar z}(\tau)\,d\tau
 =R a_{k_1}t^R|t|^\gamma+o\bigl(|t|^{R+\gamma}\bigr).
$$
Since $a_{k_1}\ne0$, we obtain
$\bar B=\Theta(|t|^{R+\gamma})$. By Lemma~\ref{thm:BoundLinPred}(a),
$\bar A/\bar B\to R$, hence also $\bar A=\Theta(|t|^{R+\gamma})$.

By eq.~\eqref{eq:z-minus-bar-z} and part~(d),
\begin{align*}
 |B-\bar B|
 &\le \lVert z-\bar z(t)\rVert+\lVert z_1-\bar z(t_1)\rVert
 =O(|t|^\alpha)=o(\bar B)\\
 |A-\bar A|
 &\le |t|\,\lVert\dot z-\dot{\bar z}(t)\rVert
       +|t_1|\,\lVert\dot z_1-\dot{\bar z}(t_1)\rVert =O(|t|^{1+\alpha-2\delta})=o(\bar A)
\end{align*}
Here $|t_1|=\Theta(|t|)$, and~\eqref{eq:alpha-choice}, together with $\Lambda>R-1$, implies
$\alpha>R+\gamma$ and $1+\alpha-2\delta>R+\gamma$. Therefore $B>0$ and
\[
 \frac{A}{B}
 =\frac{\bar A(1+o(1))}{\bar B(1+o(1))}\rightarrow R
 \mbox{~~~~~~~~as~}t\rightarrow 0
\]
After shrinking the neighborhood, $|A/B-R|<\varepsilon$, so every
other candidate ratio is at distance greater than $2\varepsilon$
from $A/B$. Thus step~3 returns the correct ratio $R$.
Since $\Pi$ is finite, a common neighborhood can be chosen for all
series, with the numerator bound fixed.

\end{proof}

We are now ready to prove Theorem~\ref{th:linear-predictor}.
We have $\bar z-z^\ast= a_{k_1} t^{k_1/c}(1+o(1))$, and hence $\|\bar z-z^\ast\|=\Theta(|t|^{k_1/c})$.
Since $\|z-\bar z\|=O(|t|^\alpha)$ and $\alpha>k_1/c$, we also have $\|z-z^\ast\|=\Theta(|t|^{k_1/c})$.
By Lemma~\ref{thm:BoundLinPred}, the ``ideal predictor'' $\hat {\bar z}={\bar z} - \tfrac c {k_1}  \dot {\bar z} t $
satisfies $\|\hat{\bar z} - z^\ast\|=O(\|\bar z-z^\ast\|^{k_2/k_1})=O(|t|^{k_2/c})$.

By the previous lemma, we can assume that the ratio $k_1/c$ produced in step 3 is the correct ratio for the Puiseux series $\bar z(\cdot)$.
Recall that our predictor is given by $\hat z=z - \tfrac c {k_1}  \dot z t$.
We then have $\hat {\bar z}={\bar z} - \tfrac c {k_1}  \dot {\bar z} t$, and so
$$
\|\hat z-\hat{\bar z}\|\le \|z-\bar z\| + \tfrac {c}{k_1}\|\dot z-\dot {\bar z}\|\,|t|\le |t|^\alpha + \tfrac {c}{k_1}|t|^{\alpha-2\delta}\,|t|
$$
$$
\|\hat z-z^\ast\|\le \|\hat{\bar z} - z^\ast\| + \|\hat z-\hat{\bar z}\| \le O(|t|^{k_2/c}) + O(|t|^\alpha)+O(|t|^{\alpha-2\delta+1}) =O(|t|^{k_2/c})
$$
since $\alpha>k_2/c$ and $\alpha-2\delta+1>k_2/c$ by the choice of $\alpha$ in eq.~\eqref{eq:alpha-choice}.
The RHS of the last expression is at most $O(\|z-z^\ast\|^{k_2/k_1})$.


\section{Corrector for corank-1 systems}\label{section:extendedsystem}
Throughout this section we assume that the corank of the system at $x^\ast$ is $\kappa=1$.

Let us now assume that we have initial point $x_\circ=(z_\circ,t_\circ)$
and predictor $\hat x=(\hat z,\hat t)$ where $\hat z$ approximates $z(\hat t)$ for some Puiseux series $z(\cdot)\in\Pi$.
The predictor step moved us away from the homotopy $h(\cdot)$;
the goal of the corrector is to go back to this homotopy.



Recall that in the classical approach we are effectively solving the system
\begin{equation}\label{eq:classical-extended}
    \begin{cases}
        h(x) &= 0\\
        t-\hat t &= 0
    \end{cases}
\end{equation}
over variables $x=(z,t)$.
Its Jacobian is
\begin{equation}
\begin{pmatrix}
h_z & h_t \\
0 & 1
\end{pmatrix}
\end{equation}
Note that if $\hat t=0$ then $x^\ast$ is a root of~\eqref{eq:classical-extended},
and the Jacobian is singular at this root (since the columns of $h_z$ are linearly dependent).
This fact prevents us from setting $\hat t=0$, since then  Newton's method may not converge.

We propose to do the following instead. Below $\beta$ is the parameter used in Theorem~\ref{th:linear-predictor}.
\begin{enumerate}
\item Set $q=\frac{[\dot z_\circ^\dag \; 1]}{\|[\dot z_\circ^\dag \; 1]\|}$.
Note that  $q^\dag$ is in the null space of $h_x(x_\circ)$, since 
$h_x(x_\circ)\cdot q^\dag=[h_z(x_\circ)\;\;h_t(x_\circ)]\cdot \begin{bmatrix}-h^{-1}_z(x_\circ)h_t(x_\circ)  \\ 1 \end{bmatrix}
=-h_t(x_\circ)+h_t(x_\circ)=0$. If $\dot z_\circ$ is undefined then take $q$ to be any unit row vector with $h_x(x_\circ)\cdot q^\dag=0$.
\item Replace system~\eqref{eq:classical-extended} with
\begin{equation}\label{eq:gap1-extended}
h[\hat x](x)\eqdef    \begin{cases}
        h(x) &\\
         q \cdot (x-\hat x) &
    \end{cases}=0
\end{equation}
\item Apply Newton's method to solve equation $h[\hat x](x)=0$ starting with $x_0=\hat x$, generating 
a sequence of points $x_1,x_2,\ldots$.
Stop once we get a point $x=x_K=(z,t)$ with $\|h(x)\|\le|t|^\beta$.
\end{enumerate}

\begin{lemma}\label{lemma:q-choice}
There exists a neighborhood $\Omega$ of $x^\ast$ and constant $\lambda>0$
such that $\|q\|=1$ and $\|(D_x h[\hat x](x))^{-1}\| \le\lambda$ for any $x_\circ,x\in\Omega$.
\end{lemma}
\begin{proof}
The Jacobian of $h[\hat x](\cdot)$ is given by
\begin{equation}
D_x h[\hat x]=
\begin{pmatrix}
h_z \;\; h_t \\
q
\end{pmatrix}
=
\begin{pmatrix}
h_z(x) \;\; h_t(x) \\
q(x_\circ)
\end{pmatrix}
\label{eq:GJLAKGHAIHGASFAGBVAJSBFKAJF}
\end{equation}
If $\hat x=x_\circ=x^\ast$ then matrix $D_x h[x^\ast]=
\begin{pmatrix}
h^\ast_z \;\; h^\ast_t \\
q^\ast
\end{pmatrix}
$ is non-singular by Assumption~\ref{assumption:main}, and matrix $[h^\ast_z \;\; h^\ast_t]$ has full rank.
Singular vectors of a matrix corresponding to singular values of multiplicity~1 depend continuously
on the matrix (by the Wedin's theorem which we used in the proof of Lemma~\ref{lemma:hz-bound})~\footnote{Singular vectors are defined up to a scalar which is a complex unit;
the continuity of these vectors holds if the scalars are chosen in a specific way. This does not affect the proof since multiplying the last row 
of the matrix in~\eqref{eq:GJLAKGHAIHGASFAGBVAJSBFKAJF} by a complex unit does not change its singular values.}.
Therefore, matrix $\begin{pmatrix}
h_z \;\; h_t \\
q
\end{pmatrix}$ depends continuously on $(x,x_\circ)$ (since $q$ is the singular vector of $h_x=[h_z\;h_t]$ corresponding to singular value $0$ of multiplicity 1).
The claim follows.

\end{proof} 

Note that the guarantee of Lemma~\ref{lemma:q-choice} can be achieved by many other choices of $q$, e.g.\ if $q$ is chosen randomly.
For the result below we assume that vector $q$ is chosen to satisfy the properties in Lemma~\ref{lemma:q-choice}
but is not necessarily in the null space of $h_x(x_\circ)$.

\begin{theorem}\label{th:gap1-corrector}
Let $\Omega$ be the neighborhood of $x^\ast$ from Lemma~\ref{lemma:q-choice} (with constant $\lambda > 0$).
There exist neighborhoods $\Omega^-\subseteq\Omega^+\subseteq\Omega$ of $x^\ast$ and constant $\beta_{\min}>0$ with the following property:
if $\hat x\in\Omega^-$ then equation $h[\hat x](x)=0$ has a unique solution $\bar x=(\bar z,\bar t)\in\Omega^+$,
and it satisfies $\|\bar x-x^\ast\|\le \lambda\|\hat x-x^\ast\|$.
Furthermore,
the sequence of points $x_0=\hat x,x_1,x_2,\ldots$ generated by  Newton's method satisfies
the following: \\
\textup{ (a)}
$\|x_k-x^\ast\|\le O(\|\hat x-x^\ast\|)$ for each $k\ge 1$.
\\
\textup{ (b)}
If $\beta>\beta_{\min}$, $k\ge \ell\eqdef\lceil 2\log_2(1+\beta)\rceil$, $x_k=x=(z,t)$ and $\|h(x)\|>|t|^\beta$
then $\|x-x^\ast\|\le \|\hat x-x^\ast\|^{\Theta(2^{k-\ell})}$.
\\
\textup{ (c)} $\|x_{k+1}-x^\ast\|\le O(\|x_{k}-x^\ast\|)$ for each $k\ge 0$.

\end{theorem}

\begin{proof}
We will apply Theorem~\ref{th:NK} for function $F(x)=h[\hat x](x)$. Since $h$ is analytic on $\Omega$, there exists constant $M>0$ such that $\|h(x)\|=\|h(x)-h(x^\ast)\|\le M \cdot \|x-x^\ast\|$ for all $x\in\Omega$.
By assumption, we have $\|(F'(x))^{-1}\|\le \lambda$ for all $x\in\Omega$.
For any $x,y\in\Omega$ we have 
$$
\|F'(x)-F'(y)\|=\left\|
\begin{pmatrix}
	h_x(x) \;-\; h_x(y) \\
0
\end{pmatrix}
\right\|
=O(\|x-y\|)
$$
since $h_x=[h_z \; h_t]$ is analytic in $\Omega$.
Thus, the first two preconditions of Theorem~\ref{th:NK} hold for some constant $L>0$.
By increasing $L$, if necessary, we can assume that $B(x^\ast,\tfrac 5{3L})\subseteq \Omega$.
Set $C=\tfrac{4}{9L}$, then $[r_-,r_+]=[\tfrac {2}{3L}, \tfrac {4}{3L}]$ in Theorem~\ref{th:NK}.
Define $\hat r=\min\{\tfrac 12 r_-,\tfrac C {\lambda M}\}$,  $\Omega^-=B(x^\ast,\hat r)$ and $\Omega^+=B(x^\ast,\tfrac 12(r_-+r_+))$,
then $\Omega^-\subset\Omega^+\subseteq\Omega$.
Now consider $\hat x\in\Omega^-$. 
Since
$\|F(\hat x)\|=\left\|\begin{pmatrix} h(\hat x) \\ 0 \end{pmatrix}\right\|=\|h(\hat x)\|\le M \cdot \|\hat x-x^\ast\|\le M\cdot \hat r$, 
we have $\left\|(F'(\hat x))^{-1}F(\hat x)\right\|
\le \lambda\cdot M\cdot \hat r\le C$
and also $B(\hat x,r_+)\subseteq B(x^\ast,\hat r+r_+)\subseteq B(x^\ast,\tfrac 5 {3L})\subseteq\Omega$.
Thus, all preconditions of Theorem~\ref{th:NK} are satisfied.

By the theorem,
equation $F(x)=0$ has a unique solution $\bar x\in B(\hat x,r)$ for any $\hat x\in\Omega^-$ and $r\in(r_-,r_+]$.
We have $B(\hat x,r_-)\subseteq B(x^\ast,\hat r+r_-)\subseteq \Omega^+ \subseteq B(\hat x,\hat r+\tfrac 12(r_-+r_+))\subseteq B(\hat x,r_+)$,
and hence equation $F(x)=0$ has a unique solution $\bar x\in \Omega^+$ for any $\hat x\in\Omega^-$.
Let us denote it as $\varphi(\hat x)$. 

Since $h[\hat x](\varphi(\hat x))=F(\varphi(\hat x))=0$ for all $\hat x\in\Omega^-$, by the implicit function theorem $\varphi$ is differentiable in $\Omega^-$,
and its derivative is given by
$$D_{\hat x}\varphi(\hat x)=-(D_x h[\hat x](x))^{-1} D_{\hat x} h[\hat x](x)=-(D_x h[\hat x](x))^{-1}
\begin{pmatrix}
0 \\ -q
\end{pmatrix}
$$
Therefore, $\|D_{\hat x}\varphi(\hat x)\|\le \|D_x h[\hat x](x))^{-1}\|\cdot 
\left\|\begin{pmatrix}
0 \\ -q
\end{pmatrix}
\right\|\le \lambda\cdot 1
$. This implies that $\|\bar x-x^\ast\|=\|\varphi(\hat x)-\varphi(x^\ast)\| \le \lambda \cdot \|\hat x-x^\ast\|$.

Next, we show properties \textup{(a)}-\textup{(b)} of the sequence $\{x_k\}_{k=0,1,\ldots}$ generated by  Newton's method.
We will denote $\Delta = \|\hat x-x^\ast\|$.
Theorem~\ref{th:NK} gives $\|x_{k+1}-\bar x\|\le \tfrac {3L}2 \|x_k-\bar x\|^2$
where $x_0=\hat x$ satisfies
 $\|\hat x-\bar x\|\le \|\hat x-x^\ast\|+\|\bar x - x^\ast\|\le\Delta+\lambda \Delta$. By shrinking $\Omega^-$, if necessary, we can make sure
that for some constant $\alpha>0$ we have
$$
\|x_{k}-\bar x\|\le  \min\{\;\Delta^{\alpha 2^{k}}\;,\;\|\hat x-\bar x\|\;\}\qquad\quad\forall k\ge 1
$$

\myparagraph{Property (a)} For any $k\ge 1$
we can write $\|x_k-x^\ast\|\le \|x_k-\bar x\|+\|\bar x-x^\ast\|
\le \|\hat x-\bar x\| + \lambda\cdot\|\hat x-x^\ast\|
\le (1+\lambda)\Delta+\lambda\Delta=(1+2\lambda)\Delta
$.

\myparagraph{Property (b)}
Suppose that $\|x_k-x^\ast\|>2\Delta^{\alpha 2^{k-\ell}}$ for $k\ge \ell$. 
Then $\|\bar x-x^\ast\|\ge \|x_k-x^\ast\|-\|x_k-\bar x\|>
2\Delta^{\alpha 2^{k-\ell}} - \Delta^{\alpha 2^k}\ge \Delta^{\alpha 2^{k-\ell}}$.
We have $h(\bar x)=0$, and hence $\bar z=\bar z(\bar t)$ for some Puiseux series $\bar z(\cdot)\in\Pi$.
This implies that $\|\bar z-z^\ast\|\le O(|\bar t|^{k_1/c})$ where $k_1,c$ are the coefficients for $\bar z(\cdot)$.
We can thus write
$$
\Delta^{\alpha 2^{k-\ell}} \le \|\bar x-x^\ast\| \le \|\bar z-z^\ast\|+|\bar t-0|<O(|\bar t|^{k_1/c})+|\bar t|<|\bar t|^\delta
$$
for some constant $\delta>0$. This implies that $\|x_k-\bar x\|\le \Delta^{\alpha 2^k}\le |\bar t|^{\delta 2^\ell}$.
Since $\ell=\lceil 2\log_2(1+\beta)\rceil\ge \log_2\beta_{\min}+\log_2 \beta$,
we have 
$\|x_k-\bar x\|\le |\bar t|^{\delta \beta_{\min} \beta}$. By choosing $\beta_{\min}$ sufficiently large, we can ensure
that 
$|\bar t|^{\delta \beta_{\min}^2}\le |\tfrac 12 \bar t|$
and
$|\bar t|^{\delta\beta_{\min}}\le |\tfrac 12 \bar t|^2$. The former condition
implies
$\|x_k-\bar x\|\le |\tfrac 12 \bar t|$, hence
the point $x_k=x=(z,t)$ satisfies $|t-\bar t|\le |\tfrac 12 \bar t|$ and so $|t|\ge |\tfrac 12 \bar t|$.
The latter condition implies that
$\|x_k-\bar x\|\le |\tfrac 12 \bar t|^{2\beta}\le|\tfrac 12 \bar t|^{\beta+\beta_{\min}}$. 
It remains to observe that $\|h(x_k)\|=\|h(x_k)-h(\bar x)\|\le O(\|x_k-\bar x\|)\le O(|t|^{\beta+\beta_{\min}})<|t|^\beta$ if neighborhood $\Omega$ is sufficiently small.

\myparagraph{Property (c)} Since the last equation of system~\eqref{eq:gap1-extended} is linear and satisfied by the initial point $\hat x$, it will be satisfied by every iterate of  Newton's method, i.e.\ $F(x_k)=\begin{pmatrix}h(x_k) \\ 0\end{pmatrix}$ for each $k\ge 0$. Therefore,  $\|F(x_k)\|=\|h(x_k)\|\le M\cdot\|x_k-x^\ast\|$.
Newton's update now gives
$\|x_{k+1}-x_k\|=\|(F'(x_k))^{-1}F(x_k)\|\le \lambda\cdot M \cdot \|x_k-x^\ast\|$
and hence 
$\|x_{k+1}-x^\ast\|\le(\lambda M+1) \cdot \|x_k-x^\ast\|$.

\end{proof}

We can finally prove Theorem~\ref{th:gap1:main}.
We use the following algorithm. Given point $x_\circ$, we compute predictor $\hat x$ as described in Section~\ref{sec:linear-predictor},
then construct system~\eqref{eq:gap1-extended} and run Newton's method, obtaining sequence
$x_0=\hat x,x_1,x_2,\ldots$. We stop once we get a point $x_k=(z,t)$ with $\|h(x)\|\le |t|^\beta$.
If $k\le\ell=\lceil 2\log_2(1+\beta)\rceil$ then we return $x_k$,
otherwise we return the sequence $(x_\ell,x_{\ell+1},\ldots,x_k)$.
In the former case we have $\|x_k-x^\ast\|\le O(\|\hat x-x^\ast\|)$ 
by Theorem~\ref{th:gap1-corrector}(a), and in the latter case $\|x_i-x^\ast\|\le O(\|\hat x-x^\ast\|^{\Theta(2^{i-\ell})})$ for each $i\in[\ell,k-1]$  by Theorem~\ref{th:gap1-corrector}(b)
and also for $i=k$ by Theorem~\ref{th:gap1-corrector}(c). Furthermore, $O(\|\hat x-x^\ast\|)=O(\|\hat z-z^\ast\|)\le O(\|z_\circ-z^\ast\|^{k_2/k_1})$ by Theorem~\ref{th:linear-predictor}.
Thus, this algorithm
has the properties stated in Theorem~\ref{th:gap1:main}.

\myparagraph{Connection to the pseudo-arclength method} The corrector described above can be related
to the pseudo-arclength method \cite{wempner1971discrete, riks1972application, keller1977numerical}.
For simplicity, we discuss it only for real-valued functions.
The method constructs the following system over variables $x=(z,t)$:
\begin{equation}\label{eq:arc-length}
    \begin{cases}
        h(x) &\\
         \frac {dz}{ds} (z-z_\circ) + \frac{dt}{ds}(t-t_\circ)-\Delta s &
    \end{cases}=0
\end{equation}
where parameter $s$ represents the Euclidean length along the curve $(z(t),t)$,
and quantities $\frac {dz}{ds}$, $\frac{dt}{ds}$, $\Delta s$ are fixed.
The first two quantities are computed from equations $\frac{dz}{ds}=\dot z \cdot \frac{dt}{ds}$ and $\left\|\frac {dz}{ds}\right\|^2 + \left|\frac {dt}{ds}\right|^2=1$.
Thus, the last equation in~\eqref{eq:arc-length} can also be equivalently written as $q\cdot (x-\hat x)=0$ for some $\hat x\in\mathbb R^{n+1}$,
where $q=[\dot z_\circ^\top\; 1]$, as in~\eqref{eq:gap1-extended}.
The difference is that we set $\hat x$ explicitly via an endgame-aware linear predictor,
while pseudo-arclength methods control parameter $\Delta s$ instead.


\section{Alternative predictors}\label{section:estimatesk1ck2c}
In practice, the power series endgame is often used with a cubic predictor,
which is more accurate than the linear predictor.
Accordingly, we conducted our experiments both with a linear and with a cubic predictor.
Details of these predictors are described below.


\subsection{Estimating $k_1/c$}\label{section:k1c-methods}
All predictors need to estimate $k_1/c$.
In this section we review three different methods for that: two existing ones and one new.
We refer to them as \cSORTcode, \cLOGcode\ and \cRATIOcode, respectively.


\myparagraph{Method 1: Trial-and-error (\cSORTcode)}
A standard, albeit heuristic, approach to estimate the winding number $c$ is to track the solution curve $z(t)$ from some initial $t_1 \in (0,1]$ down to a smaller value $t_2 < t_1$ using small steps $\Delta t$.
%
The cycle number is estimated by identifying the integer $c \in \{1, \dots, c_{\max}\}$ 
that minimizes the prediction residual $\|z_2 - \hat{z}_c(t_2)\|$, where $\hat{z}_c(t)=z_1 + c \dot z_1 t_1 \left(\left(\tfrac{t}{t_1}\right)^{1/c}-1\right)$
 is the linear predictor at $t_1$ parameterized by the candidate cycle number. Another version of the rule uses three points and a cubic predictor~\cite{MSW1992}; we do not evaluate such scheme here.


\myparagraph{Method 2: Geometric Sequence Sampling (\cLOGcode)}
An alternative method 
is to sample the path along a geometric sequence \cite{bates2011parallel, SommeseWampler}.
The ratio $k_1/c$ is estimated using three tracked points $z_1, z_2, z_3 \in \mathbb{C}^n$ evaluated at geometrically decreasing parameter values, e.g., $z_{\ell} := z(\rho^{\ell} t)$ for some step ratio $0 < \rho < 1$. 

To extract the exponents, the method subtracts the values of two consecutive points in this sequence and projects them onto a randomly chosen generic vector $v \in \mathbb{C}^n$. Taking the dot product isolates a scalar sequence:
$$\langle v, z_{\ell} - z_{\ell+1} \rangle = \left\langle v , \sum_{j=k_1}^\infty a_{j} t^{j/c} (1-\rho^{j/c}) (\rho^{j \ell/c}) \right\rangle = \sum_{j=k_1}^\infty v_{j} \rho^{j \ell/c}$$
where $v_{j} := \langle v, a_{j} \rangle t^{j/c} (1-\rho^{j/c}) \in \mathbb{C}$. Let $\Delta_{\ell} := \langle v, z_{\ell} - z_{\ell+1} \rangle$. Taking the logarithmic difference isolates the leading fractional exponent:
\begin{align*}
  \log \left| \frac{\Delta_{\ell+1}}{\Delta_{\ell}} \right| 
  &= \log \left| \frac{\sum_{j=k_1}^\infty v_{j} \rho^{j(\ell+1)/c}}{\sum_{j=k_1}^\infty v_{j} \rho^{j\ell/c}} \right| \\
  &= \log \left| \rho^{k_1/c} \frac{v_{k_1} + \sum_{j>k_1} v_{j} \rho^{(j-k_1)(\ell+1)/c}}{v_{k_1} + \sum_{j>k_1} v_{j}\rho^{(j-k_1)\ell/c}} \right| \\[1ex]
  &\approx \frac{k_1}{c} \log \rho
\end{align*}
where the approximation becomes exact as $t \to 0$ (or as $\ell$ becomes large), forcing the higher-order terms to vanish.

\myparagraph{Method 3: A ratio method (\cRATIOcode)} 
The third method that we consider was already described earlier in Section~\ref{sec:linear-predictor}. 
Let $z_1,z_2$ be two tracked points at parameters $t_1,t_2$. By Lemma~\ref{thm:BoundLinPred}, we have
\begin{equation}\label{eq:cRATIO}
  \frac{\|t_2\dot z_2-t_1\dot z_1\|}{\|z_2-z_1\|} \approx \frac{k_1}c
\end{equation}
where the approximation becomes exact as $t_1 \to 0$ and $t_2\in[0,t_1)$.

\subsection{Cubic predictor}\label{section:PredBounds}
A standard higher-order predictor used in the classical power-series endgame \cite{SommeseWampler, BatesHauensteinSommeseWampler2013}
is the cubic predictor constructed via Hermite interpolation from points $z_1,\dot z_1,z_2,\dot z_2$ at $t_1,t_2$.
Assuming that $z(t)= a_0 + a_1 s + a_2 s^2 + a_3 s^3$
where $s=s(t)=t^{1/c}$,
we obtain the following system of equations:
\begin{equation} \label{eq:hermite_matrix}
\begin{bmatrix}
1 & s_1 & s_1^2 & s_1^3 \\[0.5em]
0 & \frac{1}{c}s_1 & \frac{2}{c}s_1^2 & \frac{3}{c}s_1^3 \\[0.5em]
1 & s_2 & s_2^{2} & s_2^{3} \\[0.5em]
0 & \frac{1}{c}s_2 & \frac{2}{c}s_2^{2} & \frac{3}{c}s_2^{3}
\end{bmatrix}
\begin{bmatrix}
a_{0} \\[0.5em] a_{1} \\[0.5em] a_{2} \\[0.5em] a_{3}
\end{bmatrix}
=
\begin{bmatrix}
z_1 \\[0.5em] t_1\dot{z}_1 \\[0.5em] z_2 \\[0.5em] t_2\dot{z}_2
\end{bmatrix}
\end{equation}
One can then solve for $a_0,a_1,a_2,a_3$
and obtain approximation $\hat z(t)= a_0 + a_1 t^{1/c} + a_{2} t^{2/c} + a_{3} t^{3/c}$.

One complication that arises is that function $t^{1/c}$ has $c$ different branches.
The approximation $z(t)\approx a_0 + a_1 s + a_{2} s^{2} + a_{3} s^{3}$ will
hold only if we choose the right branch when evaluating $s=t^{1/c}$.
More precisely, we can choose any branch for $s_1=t_1^{1/c}$,
since the change $s\mapsto \omega s$ where $\omega^c=1$ can be compensated
by the change $a_{i}\mapsto \omega^{-i} a_{i}$ for all $i$.
But for the second point this argument does not work, so the predictor will be accurate
only if we choose the right branch for $s_2=t_2^{1/c}$.

The classical power series endgame works with real positive values of $t$, so there is a natural branch to choose.
This is not the case for the arclength endgame. We will discuss in the next section how we deal with this issue.

\section{Experimental results}

In this section, we evaluate the computational performance of the proposed method. 
We benchmark our approach against a classical two-stage approach
consisting of an initial \textit{predictor-corrector path tracker} followed by a standard \textit{power-series endgame}, 
as detailed in the foundational literature \cite{SommeseWampler, BatesHauensteinSommeseWampler2013}.
The description of these stages is given below.

\myparagraph{Initial stage}\label{par:initial-stage}
First, we generate a sequence of points $\{x_i=(z_i,t_i)\}$ with $1=t_0>t_1>\ldots$ 
as follows. At step $i$ we compute $t_{i+1}=\rho t_i$ for some parameter $\rho\in(0,1)$,
then estimate the ratio $k_1/c$ from $x_{i-1}$ and $x_i$ using the \cRATIOcode\ method,
then compute the linear predictor $\hat z_{i+1}$ at $t_{i+1}$,
and finally run up to $N_{\tt corr}=3$ Newton correction steps starting from $\hat z_{i+1}$
to obtain new point $z_{i+1}$. The corrector is declared convergent as soon as Newton's
updates $\Delta z^{(k)}$ contract by three orders of magnitude, $\|\Delta z^{(k)}\|/\|\Delta z^{(k-1)}\|<10^{-3}$;
if this happens within $N_{\tt corr}$ steps we accept the point. Otherwise we reject the point, update $\rho:=\rho^{0.7}$ 
and repeat the step. After 5 consecutive steps without rejections we also update $\rho:=\rho^2$.
The same corrector (with the same $N_{\tt corr}$ and convergence criterion) is used by all trackers described below.

During this process we also test whether we have entered the endgame zone.
We implemented the following criterion: we estimate the cycle number $c$ using \cRATIOcode\ and \cLOGcode\ rules
(assuming that $k_1=1$), and check whether they both agree with the same integer value $c\ge 2$ within a fixed tolerance of $10^{-2}$.
Once this happens, we enter the endgame stage described below.

This criterion worked in all our tests for corank-1 systems. In particular, during the endgame the cycle number $c$ computed from the last two points was always the same. 
(Recall that it must be unique by Proposition~\ref{prop:k1}).
Our implementation switches to the initial stage with some minimum number of steps if different $c$ is detected or if the progress of the \ALcode\ method is stalled, but this was never triggered.

\myparagraph{Power series endgame (\CLASSICcode)}
The endgame has the same structure as the initial tracker.
We tested two versions: \CLASSICcode\ with a linear predictor (\CLASSICcodeL) and 
\CLASSICcode\ with a cubic predictor (\CLASSICcodeC). They predict $\hat x_{i+1}$
using respectively one point $x_i$ and two points $x_{i-1},x_i$.
The only other difference from the initial tracker is that $\rho$ is updated more 
aggressively: we update $\rho:=\rho^\alpha$ after \emph{every} 
successful iteration (a rejected step still shrinks the step as before, $\rho:=\rho^{0.7}$).
For the linear predictor we test $\alpha\in\{1.3,1.5,1.7\}$ and for the cubic predictor $\alpha\in\{1.95,2,2.05\}$.
The cubic predictor achieves its best observed convergence rate under this aggressive schedule with $\alpha\approx 2$:
for $\alpha\in\{1.95,2\}$ the corrector converged within two Newton steps on essentially every iteration across all tested instances,
whereas for $\alpha\ge 2.05$ rejected steps appear and take many more iterations to reach the target accuracy.

\myparagraph{Arclength endgame (\ALcode)}
We again implemented two versions.
At step $i$ they set a target value $\hat t$, compute a predictor $\hat z$,
and then run the arclength corrector described in Section~\ref{section:extendedsystem}
to produce new point $x_{i+1}=(z_{i+1},t_{i+1})$.
Recall that this corrector applies Newton correction steps for the system that augments the original system $h(x)=0$ with a hyperplane passing through  $\hat x=(\hat z,\hat t)$.
The corrector step uses the same convergence criterion and $N_{\tt corr}$ as in the initial stage.~\footnote{We also tried the convergence criterion described in Section~\ref{section:extendedsystem}: stop when the current iterate $x=(z,t)$ satisfies $\|h(x)\|\le|t|^\beta$ for some constant $\beta$. We found that different instances required different values of $\beta$ for the best performance (e.g.\ $\beta\in[1.1,2]$), whereas the criterion above worked well for all instances.}

In \ALcodeL\ we set $\hat t=0$, and use the linear predictor. 
This is consistent with the theory in Sections~\ref{sec:linear-predictor} and~\ref{section:extendedsystem}.
\ALcodeL\ can thus be viewed as a practical variant sharing the predictor–corrector structure.

In \ALcodeC\ we use the cubic predictor: $\hat z$ is computed from $x_{i-1},x_{i}$ using the Hermite interpolation.
As described in Section~\ref{section:PredBounds}, this requires knowing the correct branch when computing $s_{i}=t_i^{1/c}$.
We implemented the following scheme. Assume that we know correct branches for $x_{i-1}$ and $x_{i}$,
and let $\hat z(t)=\sum_j a_j t^{j/c}$ be the approximation of the Puiseux series computed from $x_{i-1}$ and $x_{i}$
(where we know the branches for $t_{i-1}^{1/c}$ and $t_{i}^{1/c}$).
We set $\hat t=|t_i|^{1+\eta} \cdot t_i$ for some parameter $\eta>0$ and then compute the corrected point $x_{i+1}=(z_{i+1},t_{i+1})$.
We then evaluate $\hat z(t_{i+1})$ using different branches for $t_{i+1}^{1/c}$, obtaining $c$ points $\hat z^{(1)},\ldots,\hat z^{(c)}$.
The ``true'' branch should give the best approximation of $z_{i+1}$.
Accordingly, we compute distances $d_a=\|\hat z^{(a)}-z_{i+1}\|$ and find 
index $a$ with the smallest distance. If this distance is significantly smaller than the second smallest distance $d_b$ (namely, if $
\tfrac{\log (1/d_a)}{\log (1/d_b)}>1.01
$) then accept this point, and use index $a$ in the next step when predicting from $x_i$ and $x_{i+1}$. Otherwise we repeat the procedure with a smaller value of $\eta$. We use the following rule for controlling $\eta$:  
we update $\eta:=\tfrac32\eta$ if $\tfrac{\log (1/d_a)}{\log (1/d_b)}>1.1$, and $\eta:=\eta/2$ otherwise. In the beginning we initialize $\eta=0.1$. With these rules, rejections very rarely occur in our tests. Furthermore, $t$ stays very close to the real axis, and the best index $a$ is always~1.

\myparagraph{Precision} As standard in Numerical Algebraic Geometry, we use floating point computations with an adaptive precision.
The working precision is set at every step in direct proportion to the order of magnitude of the homotopy parameter:
we use $\max(5\cdot(-\log_{10}|t_i|),\, 500)$ digits of multiprecision arithmetic.

\subsection{Instances} We used polynomial systems of the following form:
\begin{equation}
    F(z) = A z + B(z)
\end{equation}
Here, $A \in \mathbb{Z}^{n \times n}$ is a dense random integer matrix of rank $n - \kappa$
and $B(z)$ is a sum of monomials of degree at least 2. 
We experimented with instances of two types: (1) diagonal, where $B_i(z)=z_i^{\alpha_i}$ with randomly sampled integers $\alpha_i > 1$;
(2) non-diagonal, where each $B_i(z)$ is a sum of a fixed number of monomials, each of which has the form 
$C \prod_{k=1}^n z_k^{\alpha_k}$, with $\sum_k \alpha_k > 1$ and such that coefficients $C$ and non-negative integer exponents $\alpha_k$ are randomly sampled from
a fixed interval.
By construction, $z^\ast=0$ is a root of $F$ of corank $\kappa$, which can be expected to be isolated due to higher-order terms $B(z)$.
The final system is obtained by applying a dense random linear transformation to both the coordinates and equations: 
$f(z)=N \cdot F(Mz)$.
We use the Newton homotopy $h(z, t) = (1-t)f(z) + t\gamma(z - z_{\tt root})$
where $\gamma$ is a random complex constant used to ensure the path avoids singularities prior to $t=0$  (the ``gamma trick''),
and $z_{\tt root}$ is chosen  as a random point sufficiently close to $z^\ast$,
so that the path tracks to~$z^\ast$.

The parameters of instances are shown next to each plot.
The tuple $(n, \kappa, c)$ gives respectively the number of variables,  the prescribed corank of the Jacobian matrix at the singular root, and the cycle number.
We also show the interval from which monomial degrees $\alpha_i$ are sampled.
Finally, for non-diagonal instances we show in parentheses the number of monomials for each equation.

\subsection{Convergence comparisons}
The comparison of the methods described earlier is given in Figure~\ref{fig:conv_comparisons}.
We only consider the endgame zone, as the initial stage is the same for all algorithms.
The \CLASSICcode\ variants are labelled by their step-size exponent, \CLASSICcodeL~$(\alpha)$ with $\alpha\in\{1.3,1.5,1.7\}$ and \CLASSICcodeC~$(\alpha)$ with $\alpha\in\{1.95,2,2.05\}$; larger exponents were also tested for both predictors and stall before the target accuracy, as the $\alpha=1.7$ and $\alpha=2.05$ curves already begin to show.

In each plot, the Y-axis shows $-\log_{10} \|z_i - z^*\|$ on a logarithmic scale, where $z_i$ is the current iterate and $z^*$ is the singular root. Straight lines in these plots correspond to superlinear convergence.~\footnote{Replacing $\|z_i\!-\!z^\ast\|$ with $f(z_i)$ gives plots that are visually very similar, so we do not show them.}

The X-axis corresponds to the number of matrix inversions in the endgame zone; this provides a measure of efficiency
which is implementation- and hardware-independent. 
We run every tracker until $-\log_{10}\|f(z_i)\|\geq 10^4$ or until $100$ matrix inversions have been spent (the residual is what the stopping rule monitors; the plots show the distance to the root).
Note that \CLASSICcode\ inverts matrices of size $n\times n$,
while \ALcode\ inverts matrices of sizes $n\times n$ and $(n+1)\times(n+1)$. The condition number of former matrices grows as $t\rightarrow 0$ while the condition number of the latter stays bounded (see Sections~\ref{sec:linear-predictor} and~\ref{section:extendedsystem}). Thus, it may actually be faster to invert matrices of size $(n+1)\times(n+1)$, since one may use fewer working digits in those computations~\cite{BHSW08,Timme21}.

In all plots, \ALcodeC\ performs the best in terms of matrix inversions, followed by \ALcodeL.

\begin{remark}
In the classical approach~\cite{SommeseWampler, BatesHauensteinSommeseWampler2013} the final iterate 
is obtained by predicting to $t=0$ using the cubic predictor. In our tests, this step decreases the distance $\|z-z^\ast\|$ but increases the residual norm $\|f(z)\|$. Note that the same step can also be done for the \ALcode\ endgame, so it is orthogonal to the methods that we study. \footnote{We tested plotting $\|\hat z-z^\ast\|$ instead
of $\|z-z^\ast\|$, where $\hat z$ is obtained by predicting current point $z$ to $t=0$. All plots looked visually very similar, except that the starting point of $*$-C methods was slightly shifted compared to $*$-L methods (since the cubic predictor gives a more accurate approximation). We do not show these plots here.}
\end{remark}

\subsection{Comparing $c/k_1$ estimation methods}
In the final section, we present plots comparing the performance of the three approaches discussed in Section~\ref{section:estimatesk1ck2c} 
to estimate the ratio $c/k_1$. Each plot follows a single path from $t=1$ into the endgame zone. The points are generated by the initial stage of the tracker
(paragraph ``Initial stage'' of Section~\ref{par:initial-stage}: $\rho$ is squared only after every five accepted steps, with no aggressive update),
and at each accepted point $(z_i,t_i)$ the three rules are evaluated on the history accumulated so far; the tracker is stopped once \cLOGcode\ and \cRATIOcode\
agree with each other and with an integer $\ge2$ to within $10^{-2}$ on three consecutive steps. The six benchmark instances shown span a range of Jacobian coranks, $\kappa\in\{2,4,7\}$ (see the individual figure captions for the exact $(n,\kappa,c/k_1)$ of each instance), allowing us to assess the robustness of 
each estimation method across singularities of different severity.

For the trial-and-error rule (\cSORTcode), we compare the predicted values across all possible cycle numbers ranging from $1$ to a predefined maximum,
$C_{\max}=16$. Note that $c/k_1=17$ for the instance in Fig.~\ref{fig:gapk_comparisons}(d); we deliberately chose $C_{\max}$ in this way
to illustrate what happens when $c/k_1>C_{\max}$. For the path-limit estimation (\cRATIOcode), the estimation relies on two consecutive points from the history.
Finally, the geometric sequence sampling method (\cLOGcode) requires three specific values $(t_1, t_2, t_3)$ that form a geometric progression such 
that $t_2 = \rho t_1$ and $t_3 = \rho^2 t_1$. Consequently, we only perform this computation when three consecutive history points satisfy this 
geometric property. If this condition is not met at a given step, the computation is skipped. In the accompanying plots, these skipped computations 
are bridged using linear interpolation to maintain the continuity of the curves.

As can be seen from the plots, \cRATIOcode\ exhibits fewer oscillations prior to reaching the correct value of $c/k_1$ than \cSORTcode\ and \cLOGcode.
We do not claim that this translates to better estimation of $c/k_1$. The main advantage of \cRATIOcode\ is that it (i) directly estimates the ratio,
and (ii) uses only two points, whereas \cSORTcode\ and \cLOGcode\ have at most one of these features.

Note that in some plots (e.g.\ in Figure~\ref{fig:gapk_comparisons}(f)) the ratio $c/k_1$ changes in the middle. This might indicate jumping to a different Puiseux series (which may indeed exist when $\kappa>1$).

\clearpage

\begin{figure}[p]
    \centering
    \begin{minipage}[t]{0.48\textwidth}
        \centering
        \includegraphics[width=\linewidth]{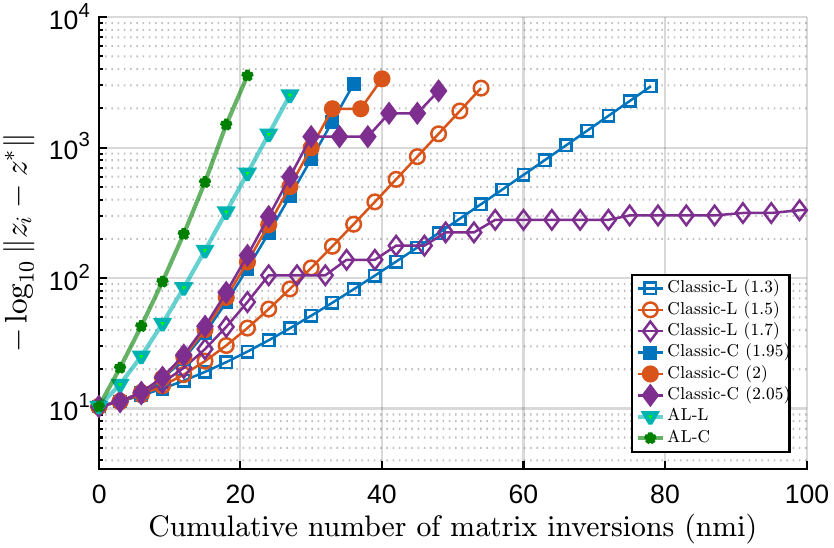}
        \par\scriptsize (a) diagonal, $(n,\kappa,c)=(12,1,4)$, $\alpha_i\in[3,12]$
    \end{minipage}\hfill
    \begin{minipage}[t]{0.48\textwidth}
        \centering
        \includegraphics[width=\linewidth]{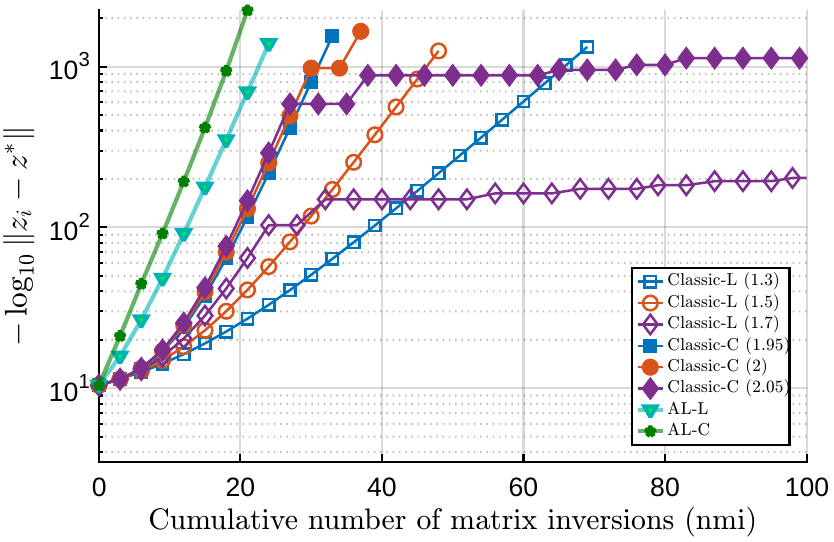}
        \par\scriptsize (b) diagonal, $(n,\kappa,c)=(5,1,8)$, $\alpha_i\in[3,15]$
    \end{minipage}
    \medskip

    \centering
    \begin{minipage}[t]{0.48\textwidth}
        \centering
        \includegraphics[width=\linewidth]{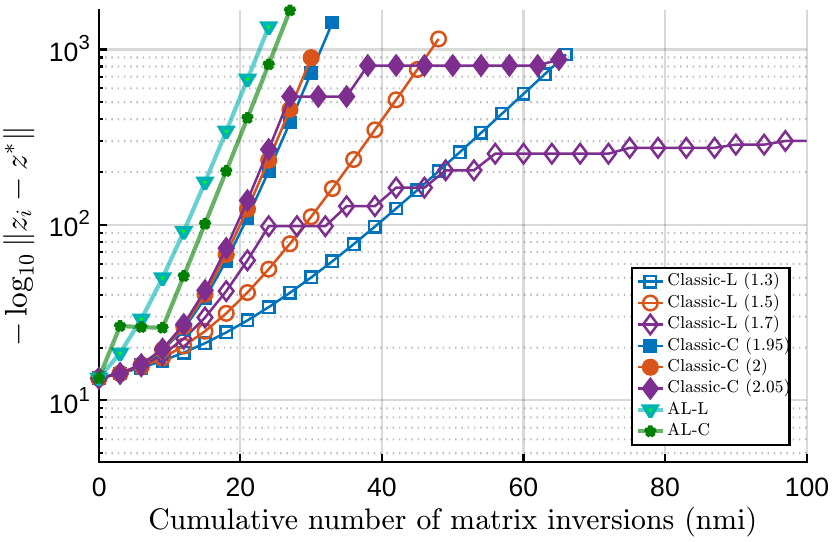}
        \par\scriptsize (c) non-diagonal (2), $(n,\kappa,c)\!=\!(5,1,12)$, $\alpha_i\!\in\![0,7]$
    \end{minipage}\hfill
    \begin{minipage}[t]{0.48\textwidth}
        \centering
        \includegraphics[width=\linewidth]{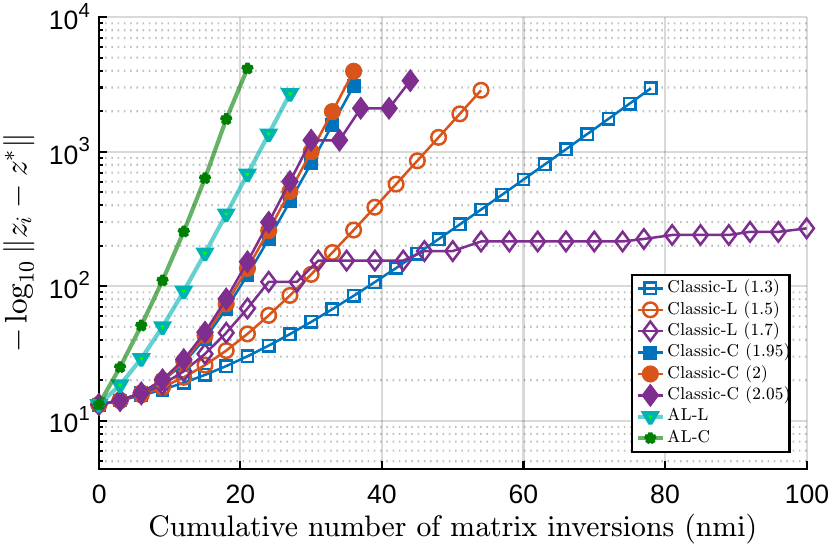}
        \par\scriptsize (d) non-diagonal (3), $(n,\kappa,c)\!=\!(4,1,4)$, $\alpha_i\!\in\![0,7]$
    \end{minipage}
    \medskip

    \centering
    \begin{minipage}[t]{0.48\textwidth}
        \centering
        \includegraphics[width=\linewidth]{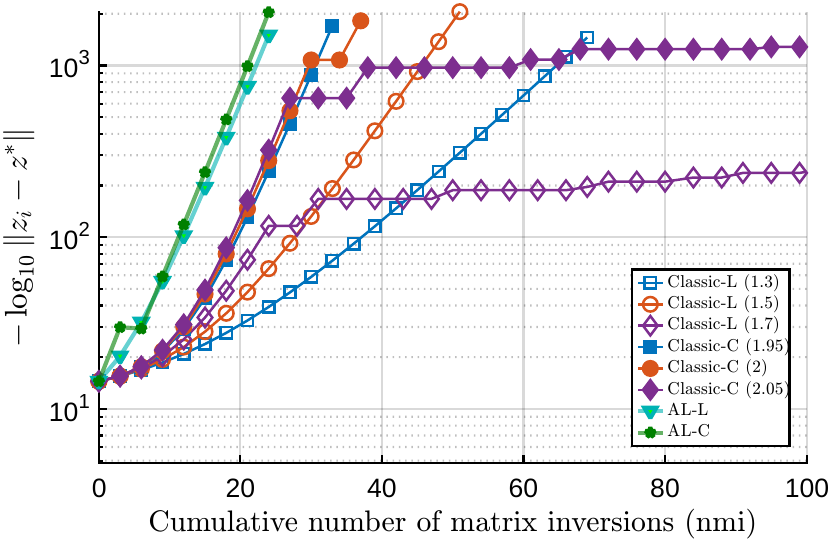}
        \par\scriptsize (e) non-diagonal (2), $(n,\kappa,c)\!=\!(7,1,7)$, $\alpha_i\!\in\![0,5]$
    \end{minipage}\hfill
    \begin{minipage}[t]{0.48\textwidth}
        \centering
        \includegraphics[width=\linewidth]{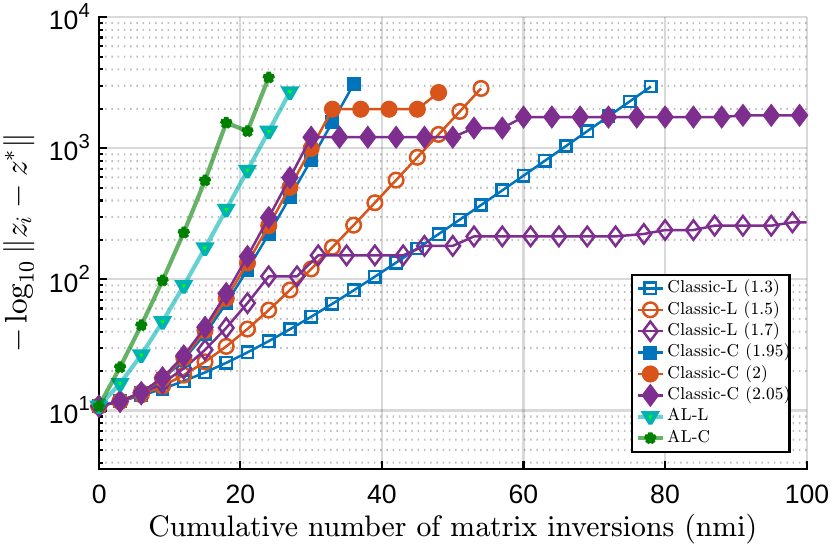}
        \par\scriptsize (f) diagonal, $(n,\kappa,c)\!=\!(6,1,4)$, $\alpha_i\!\in\![3,8]$
    \end{minipage}
    \caption{Convergence plots.
    }
    \label{fig:conv_comparisons}
\end{figure}

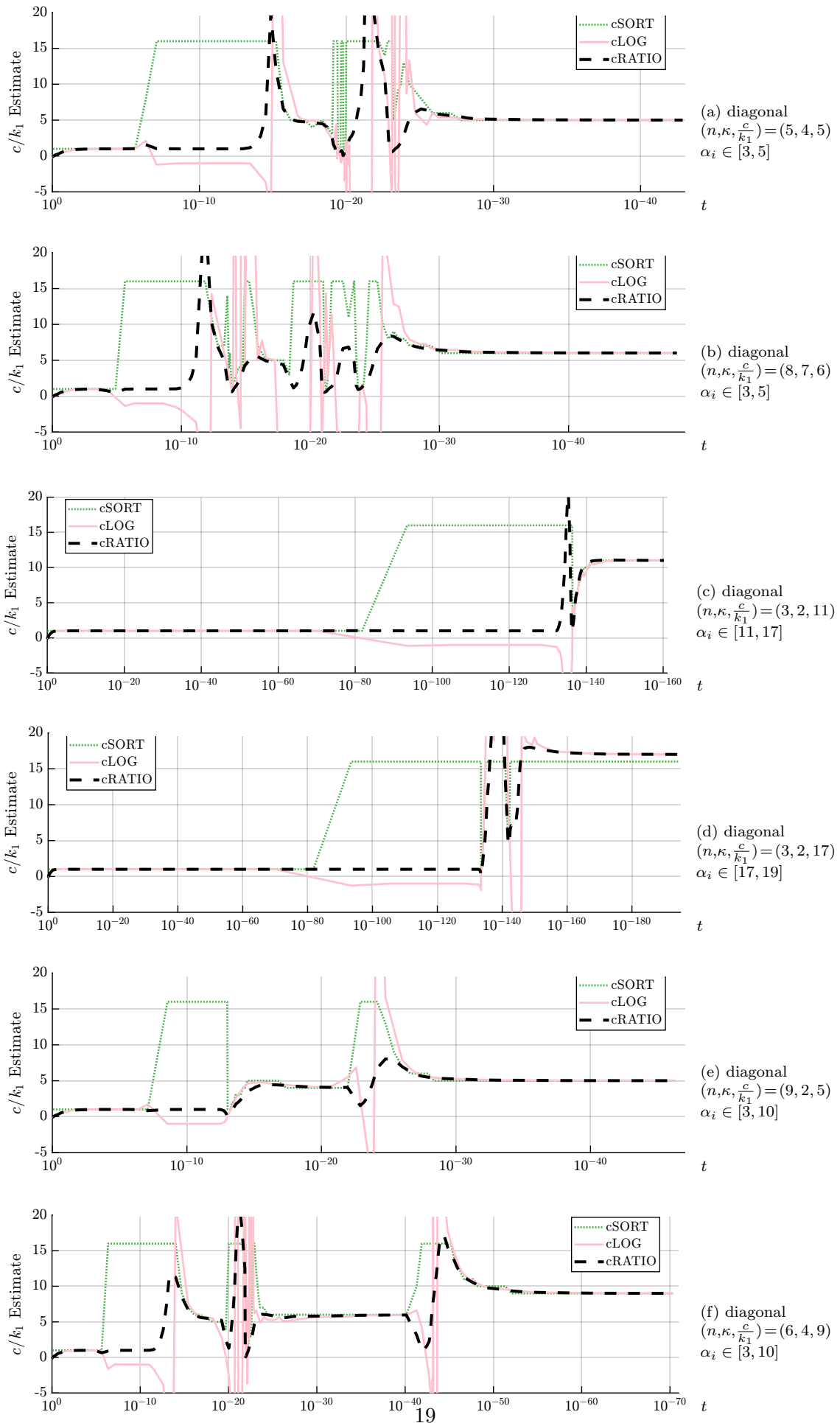
\begin{figure}[p]
\vspace{-20pt}
    \setlength{\unitlength}{0.001\textwidth}
    \newcommand{\gapplot}[1]{%
        \begin{picture}(740,259)
            \put(0,0){\includegraphics[width=740\unitlength]{#1}}
            \put(50,238){\color{white}\rule{740\unitlength}{21\unitlength}}
            \put(370,0){\color{white}\rule{40\unitlength}{18\unitlength}}
            \put(755,26){\scriptsize \color{black}$t$}
        \end{picture}%
    }
    \centering
    \begin{minipage}{\textwidth}
        \centering
        \begin{tabular}{cc}
        \gapplot{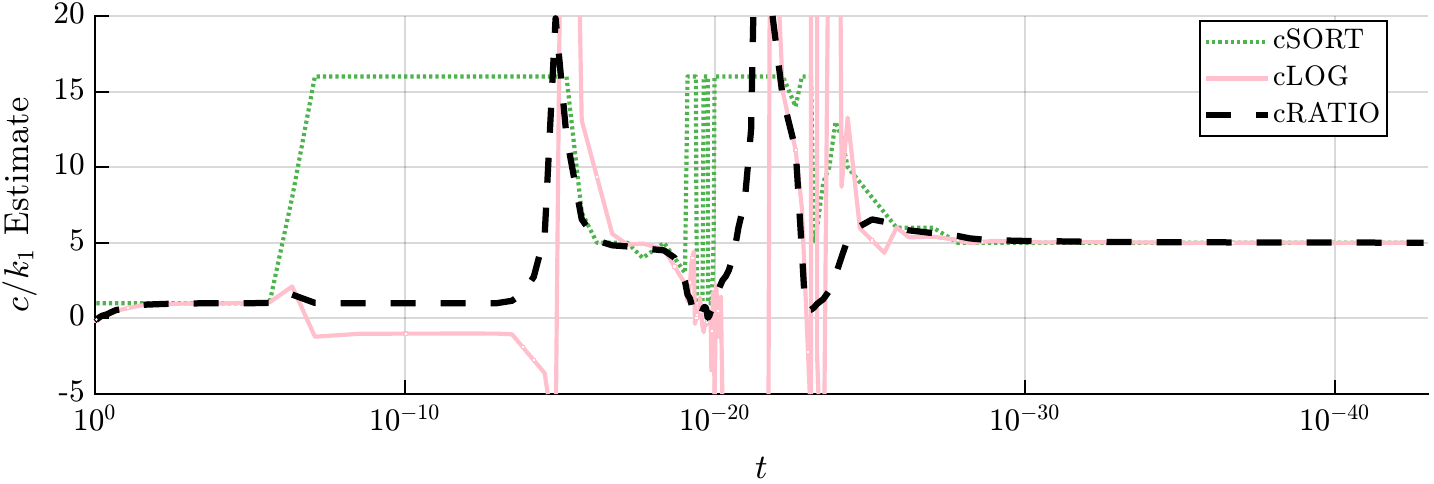}
        & \hspace{-15pt} \raisebox{50pt}{\scriptsize \begin{tabular}{l}(a) diagonal \\ $(n,\!\kappa,\!\frac c {k_1})\!=\!(5,4,5)$ \\ $\alpha_i\in[3,5]$ \end{tabular}}
        \end{tabular}
    \end{minipage}
\vspace{-5pt}

    \begin{minipage}{\textwidth}
        \centering
        \begin{tabular}{cc}
        \gapplot{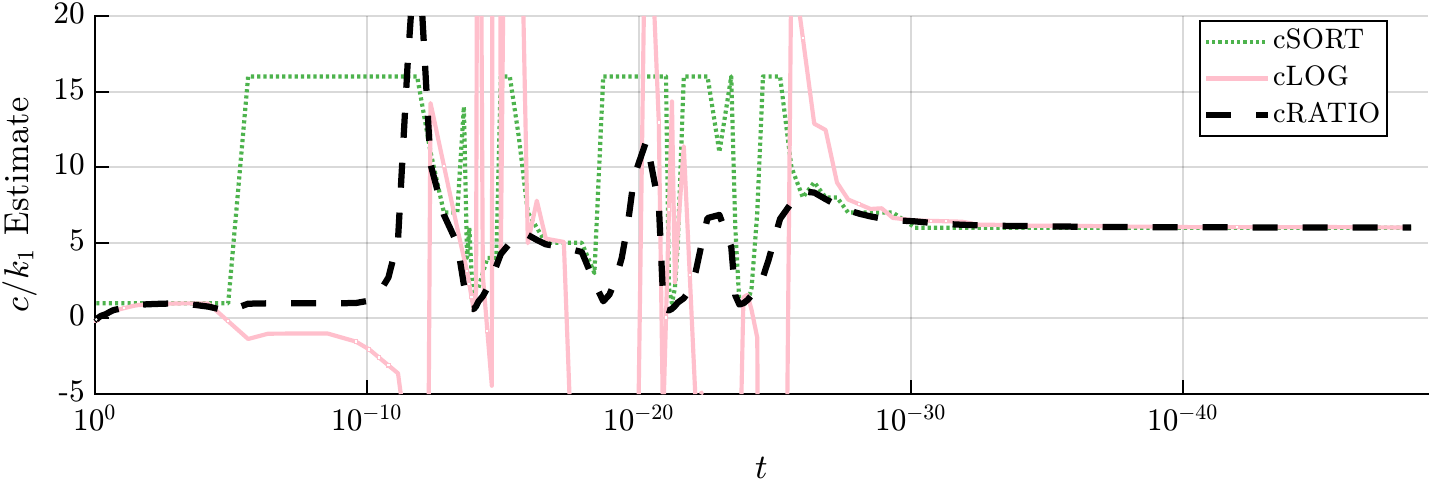}
        & \hspace{-15pt} \raisebox{50pt}{\scriptsize \begin{tabular}{l}(b) diagonal \\ $(n,\!\kappa,\!\frac c {k_1})\!=\!(8,7,6)$ \\ $\alpha_i\in[3,5]$ \end{tabular}}
        \end{tabular}
    \end{minipage}
\vspace{-5pt}

    \begin{minipage}{\textwidth}
        \centering
        \begin{tabular}{cc}
        \gapplot{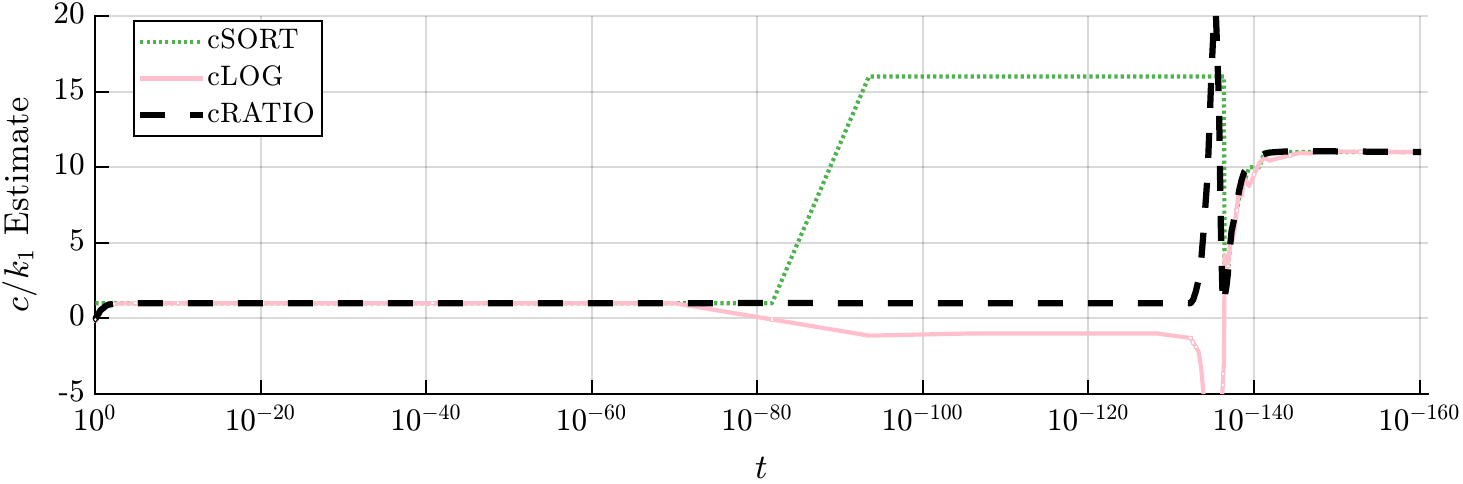}
        & \hspace{-15pt} \raisebox{50pt}{\scriptsize \begin{tabular}{l}(c) diagonal \\ $(n,\!\kappa,\!\frac c {k_1})\!=\!(3,2,11)$ \\ $\alpha_i\in[11,17]$ \end{tabular}}
        \end{tabular}
    \end{minipage}
\vspace{-5pt}

    \begin{minipage}{\textwidth}
        \centering
        \begin{tabular}{cc}
        \gapplot{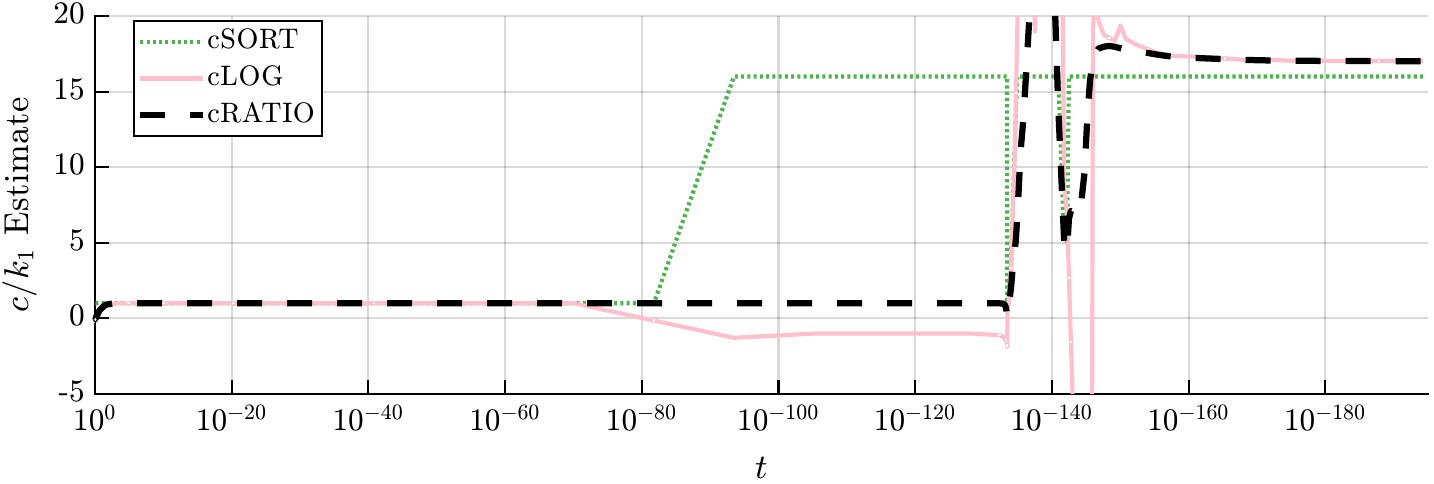}
        & \hspace{-15pt} \raisebox{50pt}{\scriptsize \begin{tabular}{l}(d) diagonal \\ $(n,\!\kappa,\!\frac c {k_1})\!=\!(3,2,17)$ \\ $\alpha_i\in[17,19]$ \end{tabular}}
        \end{tabular}
    \end{minipage}
\vspace{-5pt}

    \begin{minipage}{\textwidth}
        \centering
        \begin{tabular}{cc}
        \gapplot{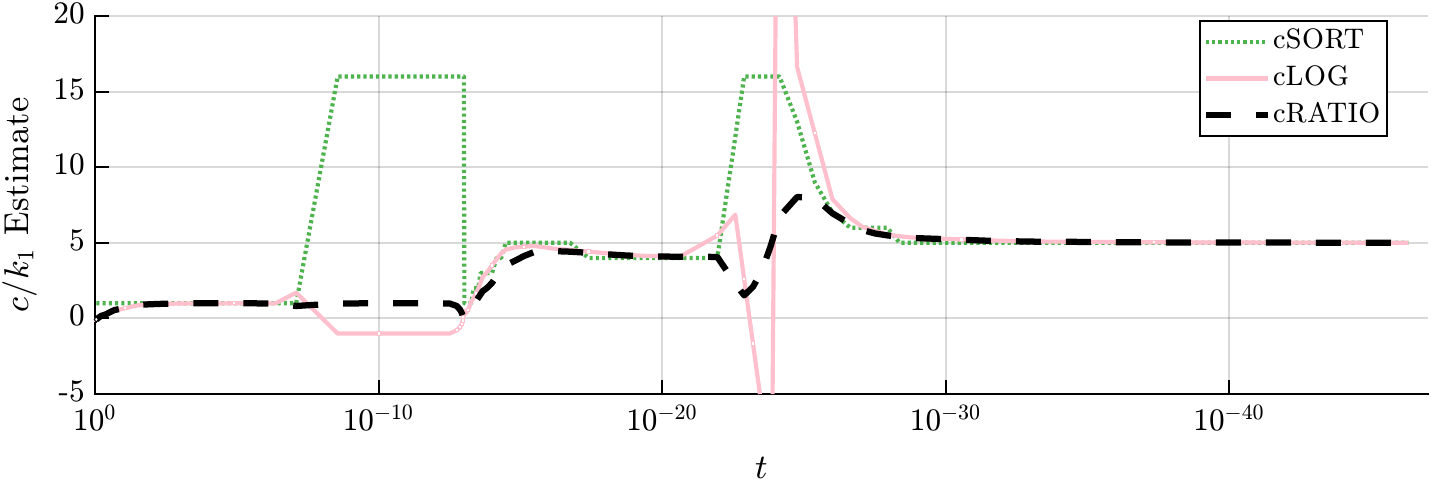}
        & \hspace{-15pt} \raisebox{50pt}{\scriptsize \begin{tabular}{l}(e) diagonal \\ $(n,\!\kappa,\!\frac c {k_1})\!=\!(9,2,5)$ \\ $\alpha_i\in[3,10]$ \end{tabular}}
        \end{tabular}
    \end{minipage}
\vspace{-5pt}

    \begin{minipage}{\textwidth}
        \centering
        \begin{tabular}{cc}
        \gapplot{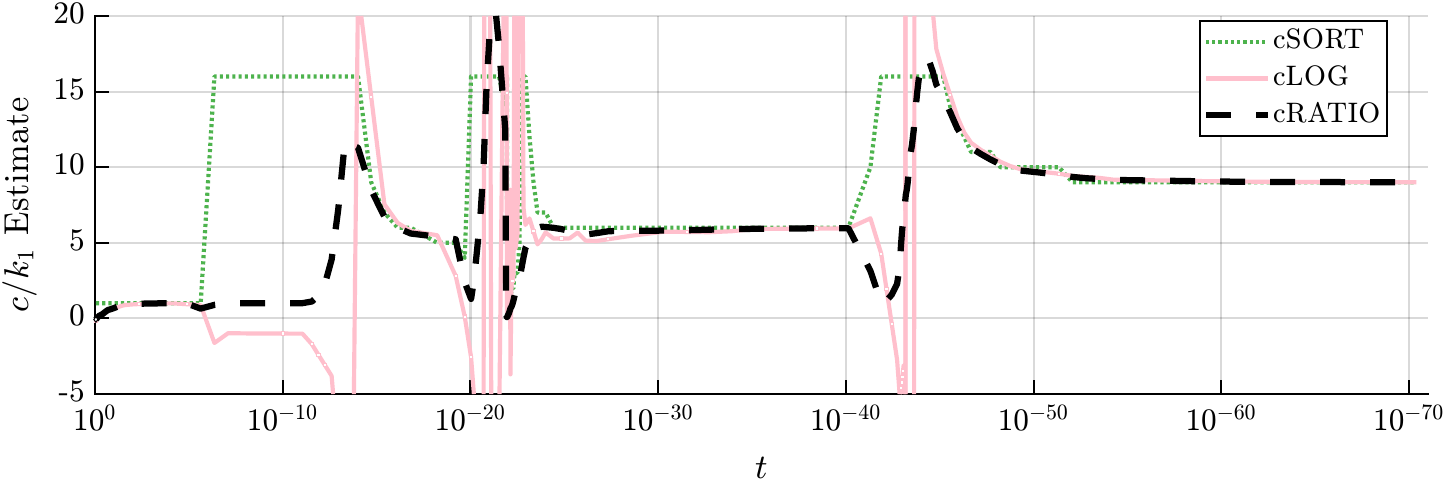}
        & \hspace{-15pt} \raisebox{50pt}{\scriptsize \begin{tabular}{l}(f) diagonal \\ $(n,\!\kappa,\!\frac c {k_1})\!=\!(6,4,9)$ \\ $\alpha_i\in[3,10]$ \end{tabular}}
        \end{tabular}
    \end{minipage}
    \caption{Comparison of the estimates of $c/k_1$.}
    \label{fig:gapk_comparisons}
\end{figure}

\clearpage

\subsection*{Declaration of generative AI usage}
The authors used ChatGPT 5.6 Sol and ChatGPT 6 Astra to check related work, proofread the manuscript and suggest corrections.
These tools identified several mistakes and typos which we then fixed.
AI also provided a proof of the part of Lemma~\ref{lemma:hz-bound} for general coranks, pointed out that $k_1=1$ under our corank-1 assumptions,
and proposed a proof of Proposition~\ref{prop:k1}.
(Before that we only showed that $k_1<c$).
All other statements and proofs were conceived and initially written without the use of AI.

All AI-assisted mathematical arguments were verified by the authors, who revised and simplified them and wrote the final proofs in their own words. 
The authors take full responsibility for the correctness and content of the manuscript.

\appendix

\section{Proof of Proposition~\ref{prop:k1}}\label{sec:prop:k1:proof}
We have ${\tt rank}([J^\ast|h^\ast_t])=n$ by assumption, so  the holomorphic implicit function theorem
implies that there exists a neighborhood $U\subseteq\mathbb C$ of $0$ and an analytic function $\varphi:U\rightarrow\mathbb C^{n+1}$
such that $\calV\cap\Omega=\varphi(U)$ for some neighborhood $\Omega$ of $(z^\ast,0)$, $\varphi(0)=(z^\ast,0)$ and $\varphi'(0)\ne 0$.
Assume that $\varphi(\cdot)=(z(\cdot),t(\cdot))$.

We claim that $t'(0)=0$ (and hence $z'(0)\ne 0$). Indeed, the assumptions imply that $w J^\ast=0$ and $w h^\ast_t\ne 0$ for some vector $w\in\mathbb C^n$.
Differentiating equation $h(z(u),t(u))$ at $u=0$ gives $J^\ast z'(0)+h^\ast_tt'(0)=0$. Multiplying the equation by $w$ on the left gives the claim.

Since $z^\ast$ is an isolated zero of $h(z,0)$ and $z'(0)\ne 0$, function $t(u)$ cannot vanish identically. 
We conclude that in some neighborhood of $0$ we have $t(u)=u^c \alpha(u)$ for some analytic function $\alpha(\cdot)$ with $\alpha(0)\ne 0$. Condition $t'(0)=0$ implies that $c\ge 2$.
Define $\beta(u)=\alpha^{1/c}(u)$ for some fixed branch of the map $\alpha\mapsto \alpha^{1/c}$,
and consider function $s(u)=u\beta(u)$ (then $t(u)=s^c(u)$). This function is analytic at $0$ and satisfies $s'(0)\ne 0$,
so there exists an inverse function $u(\cdot)$ with $u'(0)\ne 0$ which is analytic at $0$ and satisfies $s(u(s))=s$ in some neighborhood of $0$.

Denote $\tilde z(s)=z(u(s))$.
To summarize, we showed that $
\calV\cap\Omega = \{(\tilde z(s),s^c)\::\:s\in S\}
$ for some neighborhood $\Omega$ of $(z^\ast,0)$ and some neighborhood $S$ of $0$.
Function $\tilde z(s)$ is analytic at $0$ and satisfies  $\tilde z'(0)\ne 0$.
Therefore, it can be represented by a convergent series
$
\tilde z(s)=\sum_{i=0}^\infty a_i s^i
$
where $a_1 \ne 0$. Note that condition $s^c=t$ is equivalent to the condition $s=\omega t^{1/c}$ for some fixed branch $t\rightarrow t^{1/c}$
and complex $c$-root $\omega$ of $1$. This establishes the claimed structure of the set $\Pi$.

It remains to prove condition \textup{(e-ii)} of Assumption~\ref{assumption:main}.
We will show that ${\tt det}[h_z(\psi(s))]\ne 0$ for all sufficiently small $s\ne 0$,  where $\psi(s)=(\tilde z(s),s^c)$.
Suppose that the determinant is zero, then $h_{z,t}(\psi(s))\cdot \begin{pmatrix}v \\ 0\end{pmatrix}=h_z(\psi(s))\cdot v=0$ for some vector $v\ne 0$.
Differentiating equation $h(\tilde z(s),s^c)=0$ gives $h_{z,t}(\psi(s))\cdot \begin{pmatrix}\tilde z'(s) \\ cs^{c-1}\end{pmatrix}=0$.
We have ${\tt rank}(h_{z,t}(\psi(0)))=n$ by assumption, and thus ${\tt rank}(h_{z,t}(\psi(s)))\ge n$ in some neighborhood of $0$ by continuity.
Therefore, the kernel of $h_{z,t}(\psi(s))$ is at most one-dimensional. 
This implies that $\begin{pmatrix}v \\ 0\end{pmatrix}=\lambda \begin{pmatrix}\tilde z'(s) \\ cs^{c-1}\end{pmatrix}$ for some non-zero $\lambda\in\mathbb C$.
This is possible only if $s=0$.

\bibliographystyle{alpha}
\bibliography{biblio_v2.bib}

\end{document}